\newcommand\Modified{April 2016}
\documentclass[11pt]{amsart}

\usepackage{amsfonts, amsmath,amscd, amssymb, latexsym, mathrsfs, mathtools, slashed, stmaryrd, verbatim, wasysym }
\usepackage[all]{xy}
\usepackage{hyperref}

\newcommand\datver[1]{\def\datverp%
{\par\boxed{\boxed{\text{Version: #1; Run: \today}}}}}
\newcommand{\etalchar}[1]{$^{#1}$}

\newtheorem{theorem}{Theorem}

\newtheorem{definition}[theorem]{Definition}
\newtheorem{corollary}[theorem]{Corollary}

\newtheorem{lemma}[theorem]{Lemma}

\newtheorem{remark}[theorem]{Remark}

\numberwithin{equation}{section}
\numberwithin{theorem}{section}

 \usepackage{color}
\definecolor{darkgreen}{cmyk}{1,0,1,.2}
\definecolor{m}{rgb}{1,0.1,1}
\renewcommand{\bar}{\overline}

\renewcommand{\hat}[1]{\widehat{#1}}
\newcommand{\rest}[1]{\big\rvert_{#1}} % restriction e.g. to boundary

\renewcommand{\tilde}{\widetilde}
\newcommand{\wt}[1]{\widetilde{#1}}

\newcommand{\Sig}[2]{\operatorname{Sig}^{\mathrm{an}}_{#1}(#2)}
\newcommand{\Sigt}[2]{\operatorname{Sig}^{\mathrm{top}}_{#1}(#2)}
\newcommand{\Che}[2]{\Omega^{\mathrm{Che}}_{#1}(#2)}
\renewcommand{\sc}[1]{\mathbf{ #1 }^{\bullet}} %sheaf complex

\newcommand\eps\varepsilon

\newcommand\lra{\longrightarrow}
\newcommand\xlra[1]{\xrightarrow{\phantom{x} #1 \phantom{x}}}

\newcommand\pa{\partial}

\newcommand\ie{\operatorname{ie}}
\newcommand\iie{\operatorname{iie}}

\newcommand\Ie{{}^{\ie}} 
\newcommand\Iie{{}^{\iie}}
\newcommand \Z {\mathbb{Z}}
\newcommand \R {\mathbb{R}}

\newcommand\CI{{\mathcal{C}}^{\infty}}
\newcommand\CIc{{\mathcal{C}}^{\infty}_c}

\newcommand{\lrpar}[1]{\left( #1 \right)}

\newcommand\ang[1]{\left\langle #1 \right\rangle}
\newcommand{\lrbrac}[1]{\left\lbrace #1 \right\rbrace}

\DeclareMathOperator*{\btimes}{\times} %To put B under the \times symbol

\newcommand{\an}{\operatorname{an}}

\newcommand{\Ch}{\operatorname{Ch}}

\newcommand{\dR}{\operatorname{dR}}

\newcommand\id{\operatorname{id}}
\newcommand\Id{\operatorname{Id}}
\renewcommand{\Im}{\operatorname{Im}}
\newcommand{\Image}{\operatorname{Image}}
\newcommand{\ind}{\operatorname{ind}}
\newcommand{\Ind}{\operatorname{Ind}}

\renewcommand{\mid}{\operatorname{mid}}

\newcommand{\pt}{\mathrm{pt}}

\newcommand{\reg}{ \mathrm{reg} }

\newcommand{\sign}{\operatorname{sign}}

\newcommand{\tp}{\operatorname{top}}

\newcommand{\calU}{{\mathcal U}}

\newcommand{\frakS}{{\mathfrak S}}

\newcommand{\ovl}{\overline}

\newcommand\Mand{\text{ and }}

\newcommand\Mforall{\text{ for all }}

\newcommand\Mforevery{\text{ for every }}
\newcommand\Mforsome{\text{ for some }}
\newcommand\Mhence{\text{ hence }}

\newcommand\Mif{\text{ if }}

\newcommand\Min{\text{ in }}

\newcommand\Mst{\text{ s.t. }}

\newcommand\Mwhenever{\text{ whenever }}

\newcommand\Mwith{\text{ with }}

\newcommand{\wh}{\widehat}

\DeclareMathAlphabet{\mathpzc}{OT1}{pzc}{m}{it}

\newcommand\paperintro%
        {%
         }
\newcommand\paperbody%
        {%
         }

\newcommand\bbB{\mathbb{B}}
\newcommand\bbC{\mathbb{C}}

\newcommand\bbN{\mathbb{N}}

\newcommand\bbQ{\mathbb{Q}}

\newcommand\bbR{\mathbb{R}}
\newcommand\bbS{\mathbb{S}}

\newcommand\bbZ{\mathbb{Z}}

\newcommand\cA{\mathcal{A}}

\newcommand\cC{\mathcal{C}}
\newcommand\cD{\mathcal{D}}

\newcommand\cF{\mathcal{F}}

\newcommand\cH{\mathcal{H}}
\newcommand\cI{\mathcal{I}}

\newcommand\cL{\mathcal{L}}

\newcommand\cN{\mathcal{N}}
\newcommand\N{\mathbb{N}}

\newcommand\cS{\mathcal{S}}
\newcommand\cT{\mathcal{T}}
\newcommand\cU{\mathcal{U}}

\newcommand\cW{\mathcal{W}}

\newcommand\sD{\mathscr{D}}

\newcommand\sG{\mathscr{G}}

\newcommand\sW{\mathscr{W}}
\newcommand\sX{\mathscr{X}}

\newcommand\tH{\operatorname{H}}

\newcommand\tK{\operatorname{K}}

\datver{\Modified}

\begin{document}

\title{The Novikov conjecture on Cheeger spaces}

\author{Pierre Albin}
\address{University of Illinois at Urbana-Champaign}
\email{palbin@illinois.edu}
\author{Eric Leichtnam}
\address{CNRS Institut de Math\'ematiques de Jussieu}
\author{Rafe Mazzeo}
\address{Department of Mathematics, Stanford University}
\email{mazzeo@math.stanford.edu}
\author{Paolo Piazza}
\address{Dipartimento di Matematica, Sapienza Universit\`a di Roma}
\email{piazza@mat.uniroma1.it}

\begin{abstract}
We prove the Novikov conjecture on oriented Cheeger spaces whose fundamental group satisfies the strong Novikov conjecture.
A Cheeger space is a stratified pseudomanifold admitting, through a choice of ideal boundary conditions, an $L^2$-de Rham cohomology theory satisfying Poincar\'e duality.
We prove that this cohomology theory is invariant under stratified homotopy equivalences and that its signature is invariant under Cheeger space cobordism.
Analogous results, after coupling with a Mischenko bundle associated to any Galois covering, allow us to carry out the analytic approach to the Novikov conjecture: we define higher analytic signatures of a Cheeger space and prove that they are stratified homotopy invariants whenever the assembly map is rationally injective. Finally we show that the analytic signature of a Cheeger space coincides with its topological signature as defined by Banagl.
\end{abstract}

\maketitle

\tableofcontents

%%%%%%%%%%%%%%%%%%%%%%%%
\paperintro
\section*{Introduction}
%%%%%%%%%%%%%%%%%%%%%%%%

Which expressions in the rational Pontryagin characteristic classes are homotopy invariant? Novikov showed that for a simply connected smooth manifold, $X,$ the only answer is the $L$-genus, which by Hirzebruch's theorem computes the signature
\begin{equation*}
	\ang{ \cL(X), [X] } = \sigma(X).
\end{equation*}
For a non-simply connected manifold, $X,$ we can construct the `higher signatures' in term of the classifying map $r:X \lra B\pi_1X,$
\begin{equation*}
	\{ \ang{ \cL(X) \cup r^*\alpha, [X] } : \alpha \in \tH^*(B\pi_1 X;\bbQ) \}  
\end{equation*}
and the Novikov conjecture is that these are homotopy invariant. % and there are no other homotopy invariant expressions.
One very fruitful approach to this conjecture uses functional analysis to `reduce' the conjecture to the {\em strong Novikov conjecture,} recalled below, which only involves the group $\pi_1 X.$ Although the general conjecture remains open, there are now many groups for which the conjecture is known, including discrete subgroups of finitely connected Lie groups, hyperbolic groups in the sense of Gromov, and many others. 
Beyond its explicit statement, the Novikov conjecture is one of the crucial open problems in the topology of manifolds because of its close connections to many open problems in topology, differential geometry, algebra, operator algebras, and representation theory. We refer the reader to e.g., \cite{Rosenberg:Review} for a short discussion of the Novikov conjecture and the references therein for a full discussion.

Because of the importance of the conjecture and the ubiquity of singular spaces, it is natural to ask for an analogous result on some class of singular spaces.
The first obstacle is that the cohomology of singular spaces need not satisfy Poincar\'e Duality and so a singular space may not have a signature.
Goresky and MacPherson introduced the intersection homology of stratified spaces to overcome this difficulty, and showed that it satisfied a `generalized Poincar\'e Duality'. Siegel singled out the class of stratified spaces, called `Witt spaces', for which the `middle perversity' intersection homology groups satisfy Poincar\'e Duality and hence define a signature. 

Intersection homology theory is not homotopy invariant, but it is invariant under stratified homotopy equivalences. Witt spaces have $L$-classes in homology, as shown by Goresky-MacPherson and Cheeger \cite{GM1, Cheeger:Spec, Siegel}, so the Novikov conjecture here is that the higher signatures of $\hat X$ are stratified homotopy invariant. In \cite{ALMP11} the authors proved that, whenever $\hat X$ is a Witt space whose fundamental group $\pi_1\hat X$ satisfies the strong Novikov conjecture, it also satisfies the Novikov conjecture.

There is a larger class of stratified spaces that have a signature extending the classical signature. Following work of MacPherson, Cheeger, and Morgan, Markus Banagl initiated the study of self-dual sheaf complexes compatible with intersection homology in \cite{BanaglShort}. He found a class of spaces `$L$-spaces' which admit a bordism invariant signature. In \cite{ALMP13.1} the authors found an analytic analogue of Banagl's theory, christened `Cheeger spaces' 
(using the concept of ideal boundary conditions, see Definition \ref{def:mezzo})  and studied their de Rham/Hodge theory. 
In this paper, we show that these spaces satisfy

\begin{theorem}
Every Cheeger space $\hat X$ has a signature, depending only on the stratified homotopy class of $\hat X$ and invariant under Cheeger space-cobordism.
If the fundamental group of $\hat X$ satisfies the strong Novikov conjecture, then $\hat X$ satisfies the Novikov conjecture.
\end{theorem}

There are many interesting ``natural" examples of Cheeger spaces.  For example, as studied by
Banagl and Kulkarni (\cite{Banagl-Kulkarni}), reductive Borel-Serre compactifications of Hilbert modular surfaces
admit self-dual boundary conditions. It is likely that Borel-Serre compactifications of locally symmetric spaces of 
higher $\bbQ-$rank should provide other examples.

\bigskip
In more details, we start in \S\ref{sec:Recall} by recalling the resolution of a stratified space with a Thom-Mather stratification, what we term a smoothly stratified space.
It is on these spaces that we have developed the Hodge/de Rham theory in \cite{ALMP13.1}.
These cohomology groups depend on a choice of geometric data at singular strata, which we refer to as a {\em mezzoperversity} because the resulting groups lie, in a certain sense, between the `lower middle perversity' and `upper middle perversity' groups of Goresky-MacPherson.

In \S\ref{sec:TwistedDeRham} we allow twisting by a flat vector bundle and consider the resulting de Rham operators.
Importantly, since we will carry out the `analytic approach' to the Novikov conjecture (see \cite{Rosenberg:Analytic}), we consider the effect of twisting the de Rham operator by a bundle of projective finitely generated modules over a $C^*$-algebra. Indeed, if $\Gamma$ is a countable, finitely generated, finitely presented group, and $\hat X' \lra \hat X$ a $\Gamma$ covering, then the action of $\Gamma$ on $\hat X'$ and on $C^*_r\Gamma,$ the reduced $C^*$-algebra of $\Gamma,$ produce such a  flat bundle over $\hat X$. 
It is denoted by
$$
\sG(r) = \frac{\hat{X}' \times C^*_r\Gamma}{\Gamma}\, ,
$$
where $r:\hat X \lra B\Gamma$ is the classifying map of the $\Gamma$ covering. We denote   the corresponding `higher' de Rham operator by $\eth_{\dR}^{\sG(r)}.$ We prove that for any choice of mezzoperversity there is an associated domain for $\eth_{\dR}^{\sG(r)}$ which is closed, self-adjoint, and $C^*_r\Gamma$-compactly included in $L^2(X; \Lambda^* \Iie T^*X \otimes \sG(r)).$

In \S\ref{HomotopyInv} we show that the cohomology associated to a mezzoperversity, possibly twisted by a flat bundle, is a stratified homotopy invariant.
The main difficulty is that if $F: \hat X \lra \hat M$ is a stratified homotopy equivalence, the corresponding pull-back of differential forms is generally {\em not a bounded operator} on $L^2$-differential forms. We get around this by using the `Hilsum-Skandalis replacement' of $F^*,$ $HS(F),$ which is a bounded map on $L^2$-forms. The trade-off is that we must work much harder to understand how this replacement map behaves on the domains of the twisted de Rham operators.

The constructions of these sections all work in the general setting of smoothly stratified spaces, but the definition of a signature, and higher signatures, requires a more restrictive category of spaces, which we call Cheeger spaces. In section \ref{sec:Sign} we recall the definition of a Cheeger space and its signature. At first, this signature is defined using a self-dual mezzoperversity and an adapted metric.
We prove an analytic analogue of a cobordism result of Banagl \cite{Banagl:LClasses} and use it to show that the signature is independent of both the mezzoperversity and the metric. Together with the stratified homotopy invariance of the de Rham cohomology, this shows that the signature only depends on the stratified homotopy class of a Cheeger space, and is invariant under Cheeger space-cobordism. 

We also use the results of \cite{ALMP13.1} to show that the signature operator with domain given by a self-dual mezzoperversity defines a K-homology class. The 
higher signature operator $\eth_{\sign}^{\sG(r)}$ defines a class in $\tK_* (C^*_r\Gamma)$ and we show that this class, its $C^*_r \Gamma$-index, is related to the K-homology class of the `lower' signature operator via the `assembly map' $\beta.$

In \S\ref{sec:HigherSign} we follow Thom's procedure to define a homology $L$-class using the signature of submanifolds or, in our case, subspaces transversal to the stratification. We extend a result of Moscovici-Wu from the setting of Witt spaces to show that the Chern character of the K-homology class of the `lower' signature operator is equal to this $L$-class. Using this class we define the higher signatures of a Cheeger space and we prove that, whenever the assembly map $\beta$ is rationally injective, the higher signatures are stratified homotopy invariant, i.e., the Novikov conjecture. %This proves completely Theorem 1.

Finally in section \S\ref{sec:RefTop} we prove that the analytic signature of a Cheeger space coincides with its topological signature as defined by Banagl.

We rely throughout on the analytic techniques developed and explained in \cite{ALMP11}. There are other approaches to
analysis on stratified spaces, though none to our knowledge tailored so well to the types of geometric problems of interest here. 
We refer the reader to \cite{Nazarov-Plamenevsky} and references therein where closely related ideas are developed.

\medskip

\noindent {\bf Acknowledgements.} 
P.A. was partly supported by NSF Grant DMS-1104533 and an IHES visiting position and thanks Sapienza
Universit\`a di Roma, Stanford, and Institut de Math\'ematiques de Jussieu for their hospitality and support.
R.M. acknowledges support by NSF Grant DMS-1105050.
E.L. thanks Sapienza Universit\`a di Roma
for hospitality during several week-long visits;
financial support was  provided  by Istituto Nazionale di Alta Matematica (INDAM) and CNRS
through the bilateral project ``Noncommutative Geometry".
P.P. thanks the {\it Projet Alg\`ebres d'Op\'erateurs} of {\it Institut
de Math\'ematiques de Jussieu}
for hospitality during several short visits and a two months long visit
 in the Spring of 2013; financial support was
provided by Universit\'e Paris 7, Istituto Nazionale di Alta Matematica (INDAM) and CNRS
(through the bilateral project ``Noncommutative geometry")
and Ministero dell'Universit\`a e della Ricerca Scientifica
(through the project ``Spazi di Moduli e Teoria di Lie").
\\

The authors are happy to thank Markus Banagl, Jochen Br\"uning, Jeff Cheeger, Greg Friedman,
Michel Hilsum and Richard Melrose for many useful and interesting discussions.
Thanks are also due to the referee for a careful reading of the original manuscript and for interesting remarks.

%%%%%%%%%%%%%%%%%%%%%%%%
\paperbody
\section{Smoothly stratified spaces, iie metrics, and mezzoperversities} \label{sec:Recall}
%%%%%%%%%%%%%%%%%%%%%%%%

\subsection{Smoothly stratified spaces}
There are many notions of stratified space \cite{Kloeckner}.
Most of them agree that {\bf a stratified space} $\hat X$ is a locally compact, second countable,
Hausdorff 
 topological space endowed with a locally finite (disjoint) decomposition 
\begin{equation*}
	\hat X = Y^0 \cup Y^1 \cup \ldots Y^T
\end{equation*}
whose elements, called strata, are locally closed and verify the {\em frontier condition}: 
\begin{equation*}
	Y^j \cap \bar{Y^k} \neq \emptyset \implies Y^j \subseteq \bar{Y^k}.
\end{equation*}
The depth of a stratum $Y$ is the largest integer $k$ such that there is a chain of strata
$Y = Y^k, \ldots, Y^0$ with $Y^j \subset \overline{Y^{j-1}}$ for $1 \leq j \leq k$. A stratum of maximal 
depth is always a closed manifold. The maximal depth of any stratum in $\hat X$ is called the depth of $\hat X$ as 
a stratified space. Thus a stratified space of depth $0$ is a smooth manifold with no singularity.

Where the definitions differ is on how much regularity to impose on the strata and how the strata `fit together'. 
In this paper we shall mainly consider  {\it smoothly stratified pseudomanifolds} with  {\it Thom-Mather control data}.
We proceed to recall the definition, directly taken from \cite{BHS} and \cite{ALMP11}.

\begin{definition}
A smoothly stratified space of depth $0$ is a closed manifold. Let $k \in \N$, assume that 
the concept of smoothly stratified space of depth $\leq k$ has been defined.
A smoothly stratified space $\hat X$ of depth $k+1$ is a  locally compact, second countable Hausdorff space which admits a locally finite 
decomposition into a union of locally closed {\rm strata} $\frakS = \{Y^j\}$, 
where each $Y^j$ is a smooth (usually open) manifold, with dimension depending on the index $j$. 
We assume the following:

\begin{itemize}
\item[i)] If $Y^i, Y^j \in \frakS$ and $Y^i \cap \overline{Y^j} \neq \emptyset$,
then $Y^i \subset \overline{Y^j}$. 
\item[ii)] Each stratum $Y$ is endowed with a set of `control data' $T_Y$, $\pi_Y$ and $\rho_Y$;
here $T_Y$ is a neighbourhood of $Y$ in $X$ which retracts onto $Y$, $\pi_Y: T_Y \longrightarrow Y$ is 
a fixed continuous retraction and $\rho_Y: T_Y \to [0,2)$ is a  `radial function'  \footnote{In our previous paper
\cite{ALMP11}
we inexplicably required $\rho_Y$ to be proper; this is of course wrong, as it would imply  all strata to be compact} 
tubular neighbourhood such that $\rho_Y^{-1}(0) = Y$. Furthermore, we require that if $Z \in \frakS$ and 
$Z \cap T_Y \neq \emptyset$, then 
\[
(\pi_Y,\rho_Y): T_Y\cap Z \longrightarrow Y \times [0,2)
\]
is a proper differentiable submersion. 
\item[iii)] If $W,Y,Z \in \frakS$, and if $p \in T_Y \cap T_Z \cap W$ and $\pi_Z(p) \in T_Y \cap Z$,
then $\pi_Y(\pi_Z(p)) = \pi_Y(p)$ and $\rho_Y(\pi_Z(p)) = \rho_Y(p)$.
\item[iv)] If $Y,Z \in \frakS$, then
\begin{eqnarray*}
Y \cap \ovl{Z} \neq \emptyset & \Leftrightarrow & T_Y \cap Z \neq \emptyset, \\
T_Y \cap T_Z \neq \emptyset & \Leftrightarrow & Y \subset \ovl{Z}, \ Y=Z\ \  \mbox{or}\ Z \subset \ovl{Y}.
\end{eqnarray*}
\item[v)] There exist a family of smoothly stratified spaces (with Thom-Mather control data) of depth less than or
equal to $k$, indexed by $\frakS$, $\{L_Y, Y\in \frakS\}$, with the property that the restriction $\pi_Y: T_Y \to Y$ is a 
locally trivial fibration with 
fibre the cone $C(L_Y)$ over  $L_Y$ (called the link over $Y$), with atlas $\calU_Y = 
\{(\phi,\calU)\}$ where each $\phi$ is a trivialization $\pi_Y^{-1}(\calU) \to \calU \times C(L_Y)$, and the 
transition functions are stratified isomorphisms  of $C(L_Y)$ which preserve the rays of 
each conic fibre as well as the radial variable $\rho_Y$ itself, hence are suspensions of isomorphisms of 
each link $L_Y$ which vary smoothly with the variable $y \in \calU$. 
\end{itemize}

If in addition we let $\hat X^j$ be the union of all strata of dimensions less than or equal to $j$, and
require that 
\begin{itemize}
\item[vi)] $\hat X = \hat X^n \supseteq \hat X^{n-1} = \hat X^{n-2} \supseteq \hat X^{n-3} \supseteq \ldots \supseteq \hat X^0$ and
$\hat X \setminus \hat X^{n-2}$ is dense in $\hat X$
\end{itemize}
then we say that $\hat X$ is a stratified pseudomanifold. 
\end{definition}
The depth of a stratum $Y$ is the largest integer $k$ such that there is a chain of strata
$Y = Y^k, \ldots, Y^0$ with $Y^j \subset \overline{Y^{j-1}}$ for $1 \leq j \leq k$. A stratum of maximal 
depth is always a closed manifold. The maximal depth of any stratum in $\hat X$ is called the depth of $\hat X$ as 
a stratified space.

We refer to the dense open stratum of a stratified pseudomanifold $\wh{X}$ as its regular set,
and the union of all other strata as the singular set,
\[
\mathrm{reg}(\wh{X}) := \wh{X}\setminus \mathrm{sing}(\wh{X}), \qquad \mathrm{where}\qquad
\mathrm{sing}(\wh{X}) = \bigcup_{{Y\in \mathfrak S}\atop{\mathrm{depth}\, Y > 0}} Y.
\]
In this paper, we shall often for brevity refer to a smoothly stratified pseudomanifold with Thom-Mather
control data as a stratified space. When a distinction is needed we use the label topologically stratified
for a space satisfying the pared down conditions listed at the beginning of this section.

\smallskip
\noindent
{\bf Remarks.}\\
{\bf (i)} If $X$ and $X'$ are two stratified spaces, a stratified isomorphism between them is 
a homeomorphism $F: X \to X'$ which carries the open strata of $X$ to the open strata of $X'$ 
diffeomorphically, and such that $\pi_{F(Y)}^\prime \circ F = F \circ \pi_Y$, $\rho_Y^\prime 
= \rho_{F(Y)} \circ F$ for all $Y \in \frakS(X)$. \\
{\bf (ii)} 
%If $Z$ is any stratified space, then the cone over $Z$, denoted $C(Z)$, is the space $Z \times \RR^+$ 
%with $Z \times \{0\}$ collapsed to a point. This is a new stratified space, with depth one greater than $Z$
%itself. The vertex $0 := Z \times \{0\} / \sim $ is the only maximal depth stratum; $\pi_0$ is the natural
%retraction onto the vertex and $\rho_0$ is the radial function of the cone. \\
%{\bf (iii)}
%There is a small generalization of the coning construction. 
For any $Y \in \mathfrak{S}$, 
let $S_Y = \rho_Y^{-1}(1)$. This is the total space of a fibration $\pi_Y: S_Y \to Y$ with fibre $L_Y$. 
Define the mapping cylinder over $S_Y$ by $\mathrm{Cyl}\,(S_Y,\pi_Y) = S_Y \times [0,2)\, /\sim$ where 
$(c,0) \sim (c',0)$ if $\pi_Y(c) = \pi_Y(c')$. The equivalence class of a point $(c,t)$ is sometimes 
denoted $[c,t]$, though we often just write $(c,t)$ for simplicity. Then there is a stratified isomorphism 
\[
F_Y: \mathrm{Cyl}\,(S_Y,\pi_Y)  \longrightarrow T_Y;
\]
this is defined in the canonical way on each local trivialization $\calU \times C(L_Y)$ and
since the transition maps in axiom v) respect this definition, $F_Y$ is well-defined.\\
{\bf (iii)} For more on axiom (v) we refer the reader to \cite{ALMP11} and references therein.

\medskip

As an example, consider a
 stratified space with a single singular stratum $Y \subseteq \hat X$;  the control data induces a {\em smooth} fibration of $\cT \cap \{ \rho = \eps \}$ over $Y$ with fiber $Z,$ and the neighborhood $\cT$ itself is homeomorphic to the mapping cylinder of this smooth fibration. Another way of thinking about this is to note that the set $\wt X = \hat X \setminus \{ \rho \leq \eps \}$ is a smooth manifold with boundary, its boundary fibers over $Y$ with fiber $Z,$ and the interior of $\wt X$ is diffeomorphic to the regular part of $\hat X.$ We call $\wt X$ the {\em resolution} of $\hat X.$ There is a natural map
\begin{equation*}
	\beta: \wt X \lra \hat X
\end{equation*}
given by collapsing the fibers of the boundary fibration.
Changing $\eps$ yields diffeomorphic spaces, so a more invariant way of making this construction is to replace $Y$ in $\hat X$ with its inward-pointing spherical normal bundle, a process known as the `radial blow-up of $\hat X$ along $Y.$'

This process of replacing $\hat X$ by $\wt X$ is already present in Thom's seminal work \cite{Thom:Ensembles}, and versions of it have 
appeared in Verona's `resolution of singularities' \cite{Verona} and the `d\'eplissage' of Brasselet-Hector-Saralegi \cite{BHS}. These 
constructions show that any stratified space as described above (a Thom-Mather stratification) can be resolved to a smooth manifold, 
possibly with corners. Our version involves modifying $\hat X$ through a sequence of radial blow-ups to obtain a manifold with corners 
$\wt X$, the boundaries of which carry coherent iterated fibrations.  This structure was identified by Richard Melrose, and 
described carefully in \cite[\S 2]{ALMP11}:

\begin{definition}\label{def:ifs}
Let $\wt{X}$ be a manifold with corners, and enumerate its boundary hypersurfaces as $\{H_\alpha\}$ over
some finite index set $\alpha \in A$. An iterated fibration structure on $\wt{X}$ consists of the following data: 
\begin{itemize}
\item[a)] Each $H_\alpha$ is the total space of a fibration $\phi_\alpha: H_\alpha \to \wt Y_\alpha$,
where the fibre $Z_\alpha$ and base $\wt Y_\alpha$ are themselves manifolds with corners.
\item[b)] If two boundary hypersurfaces meet, i.e.\ $H_{\alpha \beta} := H_\alpha \cap H_\beta \neq 
\emptyset$, then $\dim Z_\alpha \neq \dim Z_\beta$.
\item[c)] If $H_{\alpha \beta} \neq \emptyset$ as above, and $\dim Z_\alpha < \dim Z_\beta$, then the 
fibration of $H_\alpha$ restricts naturally to $H_{\alpha\beta} = H_\alpha \cap H_\beta$ (i.e.\ the leaves of the fibration of 
$H_\alpha$ which intersect the corner lie entirely within the corner) to give a fibration of 
$H_{\alpha \beta}$ with fibres $Z_\alpha$, whereas the larger fibres $Z_\beta$ must be transverse to 
$H_\alpha$ at $H_{\alpha\beta}$. Writing $\partial_\alpha Z_\beta$ for the boundaries of these fibres at the 
corner, i.e.\ $\partial_\alpha Z_\beta := Z_\beta \cap H_{\alpha\beta}$, then $H_{\alpha \beta}$ is also the total 
space of a fibration with  fibres $\partial_\alpha Z_\beta$. Finally, we assume that the fibres
$Z_\alpha$ at this corner are all contained in the fibres $\partial_\alpha Z_\beta$, and in fact that
each fibre $\partial_\alpha Z_\beta$ is the total space of a fibration with fibres $Z_\alpha$. 
\end{itemize}
\end{definition}

\bigskip

The constructions in \cite{BHS} and \cite{ALMP11} show that any space $\hat X$ with a Thom-Mather stratification has a {\bf resolution} to a  manifold with corners $\wt X$ carrying an iterated fibration structure and a `blow-down map'
\begin{equation*}
	\beta: \wt X \lra \hat X
\end{equation*}
which is continuous, surjective, and restricts to a diffeomorphism between the interior of $\wt X$ and the regular part of $\hat X.$
There is a useful iterative structure to this resolution construction. Indeed, the closure $\bar Y$ of each singular stratum 
$Y \subseteq \hat X$ inherits the structure of a stratified space and the fibrations of the boundary faces in $\wt X$ have as 
bases the resolutions of these spaces, and the associated fibers are the resolutions of the links of the singular strata. 
For a normal stratified space (i.e., one whose links are connected) the boundary hypersurfaces of $\wt X$ are in one-to-one 
correspondence with the singular strata of $\hat X.$  If the link at a stratum is not connected, then the blow-up of that stratum
produces more than one boundary hypersurface (cf. \cite[Remark 2.4]{ALMP11}) and one should work with `collective boundary 
hypersurfaces' as in \cite{Albin-Melrose:Gps}. This makes only a notational difference, so we assume for simplicity henceforth
 that the boundary hypersurfaces of $\wt X$ and the singular strata of $\hat X$ are in one-to-one correspondence.

There are many advantages of working with $\wt X$ over $\hat X,$ especially for doing analysis (see \cite{ALMP11, ALMP13.1}).
One advantage is that the resolution $\wt X$ picks out a particularly well-behaved subset of functions on $\hat X.$
\begin{definition}\label{def:SmoothFun}
Given a smoothly stratified space $\hat X$ with resolution $\wt X,$ and a smooth manifold $M,$ let
\begin{equation*}
	\CI(\hat X, M) = \{ f \in \cC^0(\hat X, M): f \circ \beta \in \CI(\wt X, M) \}.
\end{equation*}
We abbreviate $\CI(\hat X,\bbR) = \CI(\hat X).$

If $\hat X$ and $\hat M$ are both smoothly stratified spaces, we define
\begin{equation*}
	\CI(\hat X, \hat M) = \{ f \in \cC^0(\hat X, \hat M) : f^*\CI(\hat M) \subseteq \CI(\hat X) \}.
\end{equation*}
\end{definition}
There is a partition of unity of $\hat X$ consisting entirely of functions from $\CI(\hat X)$ \cite{ABLMP}.
We point out that 
\begin{equation*}
	f \in \CI(\hat X, \hat M) \iff \exists \wt f \in \CI(\wt X, \wt M) \Mst f \circ \beta_X = \beta_M \circ \wt f
\end{equation*}
and% \footnote{\blue{This equation was previously 
% %
% \begin{equation*}
% 	f \in \CI(\hat X) 
% 	\iff f = \wt f \circ \beta \Mforsome \wt f \in \CI_{\Phi}(\wt X) = \{ \wt f \in \CI(\wt X) : \forall  H \exists f_H \in \CI(Y_H) \Mst i_H^*\wt f = \phi_H^*f_H \}
% \end{equation*}
% %
% which is wrong.}}
%
\begin{equation*}
	f \in \CI(\hat X)
	\iff f \circ \beta \in \CI_{\Phi}(\wt X) = \{ \wt f \in \CI(\wt X) : \ \forall \,  H \ \exists \, f_H \in \CI(Y_H) \Mst i_H^*\wt f = \phi_H^*f_H \}
\end{equation*}
where $i_H: H \rightarrow \wt X$ denotes the natural embedding.

\begin{definition}\label{def:SmoothBFun}
Our convention is that a manifold with corners has embedded boundary hypersurfaces (bhs). Thus if $H$ is a boundary hypersurface 
of $\wt X$ then there is a {\em boundary defining function} for $H,$ $r_H,$ meaning that $r_H$ is a non-negative smooth function 
on $\wt X,$ $H = r_H^{-1}(0),$ and $dr_H$ does not vanish on $H.$ A map between manifolds with corners $F: \wt X \lra \wt M$ is 
called a {\bf $b$-map} if, for some (hence any) choice of boundary defining functions $\{r_{i}\}$ of $\wt X$ and $\{ \rho_j\}$ of $\wt M,$ 
we have
\begin{equation*}
	F^*(\rho_j) = h_j \prod r_i^{e(j,i)} 
\end{equation*}
with $h_j$ a smooth non-vanishing function on $\wt X$ and $e(i,j) \in \bbN_0,$ for all 
$i,j$ \cite{Melrose:Conormal}, \cite[(A.11)]{Mazzeo:Edge}. 
\end{definition}

Thus if $H$ and $K$ are boundary hypersurfaces of $\wt X$ and $\wt M$, respectively, then the image of the interior of $H$ under a 
$b$-map $F$ is either contained in or else disjoint from $K,$ and moreover, the order of vanishing of the differential of $F$ 
along $H$ is constant. 

A {\bf smooth $b$-map between stratified spaces} $f \in \CI_b(\hat X, \hat M)$ is a smooth map whose lift $\wt f \in \CI(\wt X, \wt M)$ is a $b$-map of manifolds with corners.

By an {\bf isomorphism} of smoothly stratified spaces we mean an invertible smooth $b$-map whose inverse is also a smooth $b$-map

\begin{remark}
In \cite[Definition. 4, pg 258]{ALMP11}  our notion of smooth map between stratified spaces is what here we are calling a smooth $b$-map between stratified spaces.

As explained in \cite{ALMP11}, a smoothly stratified space is a stratified space whose cocycles take values in stratified isomorphisms.
\end{remark}

We recall other classes of maps we will use so as to have all of their definitions in one place.
\begin{definition}\label{def:StratPreserving}
A continuous map between topologically stratified spaces $f:\hat X\lra \hat M$ is {\bf stratum-preserving} if
\begin{equation*}
	T \in \cS(\hat M) \implies f^{-1}(T) = \bigcup S_i, \quad S_i \in \cS(\hat X).
\end{equation*}
Equivalently if the image of a stratum of $\hat X$ is contained in a stratum of $\hat M.$
\end{definition}
A smooth $b$-map between smoothly stratified spaces is necessarily stratum preserving.

\begin{definition}\label{def:CodPreserving} [{cf. \cite[Def. 4.8.4]{Kirwan-Woolf}}]
A stratum preserving continuous map between topologically stratified spaces $f:\hat X\lra \hat M$ is {\bf codimension-preserving} if
\begin{equation*}
	T \in \cS(\hat M) \implies \mathrm{codim} f^{-1}(T) = \mathrm{codim} T.
\end{equation*}
If $\hat X$ and $\hat M$ are smoothly stratified, we denote the set of smooth codimension preserving $b$-maps between them by $\CI_{b,cod}(\hat X, \hat M).$

A {\bf stratified homotopy equivalence} between two topologically stratified spaces is a pair of codimension preserving maps $f: \hat X \lra \hat M$ and $g: \hat M \lra \hat X$ such that $g\circ f$ and $f\circ g$ are homotopic to the identity through codimension-preserving homotopies $\hat X \times [0,1] \lra \hat X$ and $\hat M \times [0,1] \lra \hat M,$ respectively. 
A {\bf smooth stratified homotopy equivalence} between two smoothly stratified spaces is a stratified homotopy equivalence with $f \in \CI_{b,cod}(\hat X, \hat M),$ $g \in \CI_{b,cod}(\hat M, \hat X)$ and with $f\circ g$ and $g \circ f$ homotopic to the identity through smooth codimension preserving $b$-maps.
\end{definition}
This notion of stratified homotopy equivalence is that used in \cite[Proposition 2.1]{Friedman:StratFib}, where they were shown to preserve intersection homology. We will show below that smooth stratified homotopy equivalence preserve refined intersection homology groups.

We will show in the next section that smoothly stratified spaces admit a smooth stratified homotopy equivalence if and only if they admit a continuous stratified homotopy equivalence.

\subsection{Approximation of continuous functions}

Another advantage of working with $\wt X$ is that we can make use of all of the standard machinery of smooth differential geometry.
Since our object of study is $\hat X$ it is important to make constructions consistent with the iterated fibration structure of $\wt X.$
Thus, an {\bf iterated wedge metric} (or {\bf iterated incomplete edge} or {\bf iie} metric) is a Riemannian metric on $\wt X$ that degenerates at each boundary hypersurface by collapsing the fibers.
Inductively, an $\iie$ metric $g$ on a smooth manifold is a Riemannian metric, and on a general manifold $\wt X$ with an iterated fibration structure is one that near the boundary hypersurface $H$ has the form
\begin{equation*}
	dx^2 + x^2g_{\wt Z} + \phi_H^*g_{Y_H}
\end{equation*}
where we have trivially extended the fibration $\phi_H$ to a collar neighborhood of $H.$
Here $x$ is a boundary defining function for $H,$ meaning a smooth non-negative function on $\wt X$ vanishing linearly on its zero set, $H,$ $g_{Y_H}$ is a metric on the base of the fibration and $g_{\wt Z}$ restricts to each fiber of $\phi_H$ to be an $\iie$ metric on $\wt Z.$ (See \cite{ALMP11} for more details; in this paper we will work exclusively with {\em rigid} $\iie$ metrics like these without further comment.) 
% For later reference, we record also the associated metric
% \begin{equation*} \label{eq:ie}
% \frac{dx^2 }{ x^2}+ g_{\wt Z} + \frac{\phi_H^*g_{Y_H}}{x^2}\,.
% \end{equation*}
% This is `partially' complete, i.e., complete in the $x$-direction but not in $Z$ except in the depth $1$ case where $Z$ is smooth. 

Using $\iie$ metrics to measure distances, we next show that appropriate continuous maps can be approximated by smooth $b$-maps 
in the same homotopy class.

\begin{theorem} \label{thm:SmoothApp}
Let $\hat X$ and $\hat M$ be two smoothly stratified spaces and
\begin{equation*}
	h_0:\hat X \lra \hat M 
\end{equation*}
a continuous stratum-preserving map between them. For any $\eps>0,$ there is a smooth $b$-map $h_1 \in \CI_b(\hat X, \hat M)$
and a continuous homotopy
\begin{equation*}
	H: \hat X \times [0,1] \lra \hat M
\end{equation*}
 such that:\\
i) $H_0 = h_0,$ $H_1 = h_1,$ \\
ii) $H_t$ is stratum-preserving for each $t \in [0,1]$, \\ iii) If $h_0$ preserves the codimension of each stratum, 
then so does each $H_t$, $t\in [0,1]$.
\end{theorem}

\begin{proof}   We first assume that $\hat X$ and $\hat M$ are of depth one.  Choose radial functions $x$ and $x'$ on
each of these spaces near their respective singular strata, $Y$ and $W$, and fix the conic fibrations $\pi$ and $\pi'$ in each of 
the tubular neighborhoods $\mathcal U = \{x < 1\}$ and $\mathcal V = \{x' < 1\}$.  We can cover each of these neighborhoods 
by a finite number of coordinate charts, possibly scaling $x$ and $x'$ if necessary, so that we can represent $h_0$ via 
local adapted coordinates $(x,y,z)$ in $\mathcal U$ and $(x',y',z')$ in $\mathcal V$ as
$$
h_0(x,y, z) = ( x'(x,y, z) , y'(x,y, z),  z'(x,y, z) ). 
$$ 
%(Strictly speaking, we do not need that $z$ is a coordinate on the link $Z$ of the conic fibers, but just a point in $Z$.)
The functions $x'$ and $y'$ are continuous for $0 \leq x \leq 1$; furthermore, $x'(0,y,z) \equiv 0$, and the fact that $h_0$ is 
continuous on $\hat X$ means that $y'(0,y,z)$ is independent of $z$ (we are implicitly using this last fact to say that the image 
of any conic fiber $\pi^{-1}(y)$ remains within a given coordinate chart in $\wt M$). However, in general, $z'(x,y,z)$ need not 
even have a limit as $x \to 0$

We construct an initial homotopy to a map which maps the conic fibers of $\pi$ to those of $\pi'$ and so that 
$z'$ is continuous at $x=0$. Choose a smooth nonnegative function $\chi(t,x)$, $0 \leq t, x \leq 1$, such that 
$\chi(0,x) \equiv x$, $\chi(t,x) > 0$ for all $x \geq 0$ and $\chi(t,x) = x$ for $x \geq 1/2$ when $t > 0$, and 
finally $\chi(1,x) = 1/4$ for $x \leq 1/4$, and choose another such function $\tilde{\chi}(t,x) \geq 0$, again with 
$\tilde{\chi}(0,x) \equiv x$ but now $\tilde{\chi}(1,x) = 0$ for $x \leq 1/4$. Then  
$$
h_t' = ( x'(x,y,z), y'(\tilde{\chi}(t,x), y, z), z'(\chi(t,x),\,y, z))
$$ 
is a homotopy of continuous stratum-preserving maps, and $h_1'$ has the required properties. Strictly speaking
this is defined relative to coordinates, but because we have fixed $\pi$ and $\pi'$, the conic fibers are well-defined
and this procedure makes sense globally. Note that $h_1'$ lifts to a continuous map $\wt X \to \wt M$ 
which respects the fibration structure of the boundaries. We may also, at this stage, choose a further homotopy which
acts only on the radial variables and which results in a map for which $x' = x$ for $x \leq 1/4$.

We now construct the remainder of the homotopy using standard mollification operators in such local coordinate charts.
Note that by construction $x' = x$ and $y'(x,y,z)$ is constant for $x \leq 1/4$, so we need only mollify the fibre-preserving
map $h_0$ along the hypersurface $\{x = 1/4\}$. Locally this amounts to mollifying the family of functions $z'(1/4, y, z)$.
We invoke the argument detailed in \cite[Ch.2\, Thm.\ 2.6]{Hirsch} to pass from mollification in these local neighborhoods 
to a global smooth approximation, which proceeds by inductively extending the approximation over elements of a locally
finite open cover. 

Note that this whole argument only depends on the fact that $\wt X$ has depth $1$. Indeed, suppose that $\wt M$ has depth $\ell$.
We may as well assume that the image of the principal open stratum of $\wt X$ lies in the principal open stratum of $\wt M$.
In addition, the singular stratum $Y$ of $\wt X$ is a closed manifold, hence its image lies entirely within a stratum $W$ of $\wt M$,
and does not intersect any other strata.  Therefore, $\wt M$ effectively has depth $1$ near the image of $\wt X$, and since
the argument above is local, it applies directly.

We now induct on the depth of $\wt X$.   Suppose that the result has been proved for all stratified spaces with 
depth less than $k$, and for all stratified images $\wt M$. Suppose $\wt X$ has depth $k$. As before, let $Y$ denote 
its highest depth stratum and $W$ the stratum in $\wt M$ which contains $h_0(Y)$ but not $h_0(\wt X)$. Also 
choose tubular neighborhoods, fibrations and radial variables $(\mathcal U, \pi, x)$ and $(\mathcal V, \pi', x')$, respectively.  
We first choose a homotopy to a map $h_0'$ which maps the fibers of $\pi$ to the fibers of $\pi'$, satisfies $(h_0')^* (x') = x$, 
and whose projection to a map of links $Z \to Z'$ is constant in $x$ for $x \leq 1/4$.  Now apply the inductive hypothesis and 
the local nature of the proof to homotope $h_0'$ in the region $x \geq 1/4$ to a smooth stratum-preserving map $h_0''$ which 
continues to preserves the fibers of the stratified subspace $\{x = 1/4\}$. We finally choose a homotopy of the
family of maps on the links of these conic fibers which is smooth. This completes the proof. 
\end{proof}

\begin{corollary}
If $g_1: \hat X \lra \hat M$ and $g_2: \hat X \lra \hat M$ are two smooth stratum-preserving maps and they are homotopic through continuous stratum-preserving maps, then they are homotopic through smooth stratum-preserving $b$-maps.
\end{corollary}

\begin{corollary}
If $\hat X$ and $\hat M$ are smoothly stratified spaces and there is a continuous stratified homotopy equivalence between them, then there is a smooth stratified homotopy equivalence between them.
\end{corollary}

Miller \cite{Miller} established a criterion for recognizing if certain maps between two {\em homotopically stratified} spaces is a stratified homotopy equivalence. His criterion applies to smoothly stratified spaces, and our theorem shows that it recognizes smooth homotopy equivalences.

\begin{corollary}
Let $\hat X$ and $\hat M$ be smoothly stratified spaces and $g\in \CI(\hat X, \hat M).$ Assume that $\hat X$ and $\hat M$ have the same number of strata, the pre-image of a stratum of $\hat M$ under $g$ is a stratum of $\hat X$ and the induced map between the partially order sets (corresponding to the partial ordering of the strata) is an isomorphism. Then $g$ is a homotopy equivalence (through smooth maps with these properties) if and only if the induced maps on strata and links are stratified homotopy equivalences. 
\end{corollary}

\subsection{Mezzoperversities}

Iterated incomplete edge ($\iie$) metrics and the corresponding dual metrics are degenerate on $T\wt X$ and $T^*\wt X,$ respectively. 
However there are rescaled versions of these bundles on which these are {\em non-degenerate} bundle metrics.
An efficient description of the {\bf iie-cotangent bundle} $\Iie T^*X$ is given by its space of sections: 
\begin{equation*}
	\CI(\wt X; \Iie T^*X) = \{ \omega \in \CI(\wt X; T^*\wt X) : |\omega|_g \text{ is pointwise bounded} \}.
\end{equation*}
Thus in local coordinates near $H$ it is spanned by
\begin{equation*}
	dx, \quad
	dy, \quad
	x dz
\end{equation*}
where $y$ is a local coordinate along $Y_H$ and $z$ is a local coordinate along the fibers of $\phi_H.$
The {\bf iie-tangent bundle} $\Iie TX$ is the dual bundle of $\Iie T^*X.$ 

There is an analogous construction of `iterated edge', or $\ie$, bundles and metrics. Briefly, a vector field on $\wt X$ is an 
$\ie$-vector field if its restriction to any boundary hypersurface is tangent to the fibers of the boundary fibration and the 
$\ie$-vector fields are the sections of vector bundle, the {\bf ie-tangent bundle $\Ie TX.$} For more details on the $\ie$ 
and $\iie$ tangent bundles, we refer the reader to \cite[\S 4.4]{ALMP11}. This $\ie$-tangent bundle will play an important role 
in section \ref{sec:HSmaps}.

An important fact is that $\Iie T^*X$ restricted to the interior of $\wt X$ is canonically isomorphic to the cotangent bundle of $X,$ the regular part of $\hat X.$
Thus when studying differential forms on $X$ we can equally well work with sections of $\Lambda^*(\Iie T^*X).$
This has several advantages when analyzing the behavior of the exterior derivative $d$ and the de Rham operator $\eth_{\dR} = d+\delta$ of an $\iie$-metric. For example it brings out some remarkable symmetry and inductive structure of these operators at each boundary hypersurface. This is exploited in \cite{ALMP11} and \cite{ALMP13.1} to study geometrically natural domains for these differential operators when acting on $L^2.$\\

Let $(\hat X, g)$ be a stratified space with an $\iie$ metric, and let $L^2(X;\Lambda^* \Iie T^*X)$ denote the Hilbert space of differential forms on $X$ with respect to the induced measure. The operator $\eth_{\dR} = d+\delta$  is initially defined as a linear operator
\begin{equation*}
	\eth_{\dR}: \CIc(X;\Lambda^* \Iie T^*X) \lra \CIc(X;\Lambda^* \Iie T^*X)
\end{equation*}
and has two canonical extensions to a closed unbounded operator on $L^2.$
The first is the minimal domain
\begin{multline*}
	\cD_{\min}(\eth_{\dR})
	= \{ \omega \in L^2(X;\Lambda^* \Iie T^*X) : \\
	\exists (\omega_n) \subseteq \CIc(X;\Lambda^* \Iie T^*X) \Mst \omega_n \xlra{L^2} \omega \Mand (\eth_{\dR}\omega_n) \text{ is $L^2$-Cauchy} \}
\end{multline*}
and the second is the maximal domain
\begin{equation*}
	\cD_{\max}(\eth_{\dR})
	= \{ \omega \in L^2(X;\Lambda^* \Iie T^*X) : 
	\eth_{\dR}\omega \in L^2(X;\Lambda^* \Iie T^*X) \},
\end{equation*}
where $\eth_{\dR}\omega$ is evaluated distributionally.
These are adjoint domains and if they coincide we say that $\eth_{\dR}$ is essentially self-adjoint.

There are two obstructions to these domains coinciding.
The first is geometric. An $\iie$ metric $g$ on $\hat X$ induces an $\iie$ metric on each link $Z$ of $\hat X,$ and so a corresponding de Rham operator $\eth_{\dR}^Z.$
Often a closed domain for $\eth_{\dR}$ will induce a closed domain for each $\eth_{\dR}^Z$ and the {\em small eigenvalues} of $\eth_{\dR}^Z$ obstruct essential self-adjointness of $\eth_{\dR}.$ Fortunately it is easy to avoid the non-zero small eigenvalues by working with `properly scaled' metrics, see \cite[\S 5.4]{ALMP11}, \cite[\S 3.1]{ALMP13.1}, and we always assume that we have done this.

The second obstruction is topological. For example, if all of the links are odd-dimensional then for a suitably scaled $\iie$ metric the de Rham operator is essentially self-adjoint. The precise condition needed was formulated by Siegel \cite{Siegel}. A {\bf Witt space} is a pseudomanifold $\hat X$ such that every even-dimensional link $Z$ satisfies
\begin{equation*}
	IH^{\bar m}_{\dim Z/2}(Z) = 0
\end{equation*}
where $IH^{\bar m}_*$ is the upper middle perversity intersection homology of Goresky-MacPherson \cite{GM1}. These are precisely the spaces where a suitably scaled $\iie$ metric has essentially self-adjoint de Rham operator.\\

Suppose that $Y^1, \ldots, Y^T$ are the singular strata of $\hat X$ organized with non-decreasing depth, so that the link of $\hat X$ at $Y^1$ is a smooth manifold $Z^1,$ say of dimension $f^1.$ We denote by $H^i$ the boundary hypersurface of $\wt X$ lying above $Y^i.$
The $\iie$ metric $g$ on $\hat X$ induces a Riemannian metric on each $Z^1$ and we say that {\em $g$ is suitably scaled at $Y^1$} if each $\eth_{\dR}^{Z^1}$ does not have non-zero eigenvalues with absolute value less than one.
For these metrics, in \cite{ALMP13.1} we prove that any element $u \in \cD_{\max}(\eth_{\dR})$ has an asymptotic expansion at $Y^1,$
\begin{equation*}
	u \sim x^{-f^1/2}(\alpha_1(u) + dx \wedge \beta_1(u)) + \wt u
\end{equation*}
in terms of a boundary defining function $x$ for $H^1,$ however the terms in this expansion are distributional,
\begin{equation*}
	\alpha_1(u), \beta_1(u) \in H^{-1/2}(Y^1; \Lambda^*T^*Y^1 \otimes \cH^{f^1/2}(H^1/Y^1)), \quad
	\wt u \in xH^{-1}(X; \Lambda^*\Iie T^*X).
\end{equation*}
Here $\cH^{f^1/2}(H^1/Y^1)$ is the bundle over $Y^1$ with typical fiber $\cH^{f^1/2}(Z^1),$ the Hodge cohomology group of $Z^1$ with its induced metric. It has a natural flat connection discussed in detail below in \S\ref{sec:GMConn} and it inherits a metric.
This is true for all values of $f^1,$ with the convention that if $f^1$ is not even then $\cH^{f^1/2}(H^1/Y^1) =0,$ and so $\alpha_1(u) = \beta_1(u) =0.$

The stratum $Y^1$ is a {\em Witt stratum} of $\hat X$ if the group $\cH^{f^1/2}(Z^1)$ is trivial.
Otherwise the distributional differential forms $\alpha(u),$ $\beta(u)$ serve as `Cauchy data' at $Y^1$ which we use to define {\em Cheeger ideal boundary conditions}. 
For any subbundle
\begin{equation*}
\xymatrix{ W^1 \ar[rd] \ar[rr] & & \cH^{f^1/2}(H^1/Y^1) \ar[ld] \\ & Y^1 & }
\end{equation*}
that is parallel with respect to the flat connection, we define 
\begin{multline}\label{eq:DefDmaxW1}
	\cD_{\max, W^1}(\eth_{\dR}) = \{ u \in \cD_{\max}(\eth_{\dR}) : \\
	\alpha_1(u)  \in H^{-1/2}(Y^1; \Lambda^*T^*Y^1 \otimes W^1), \quad
	\beta_1(u) \in H^{-1/2}(Y^1; \Lambda^*T^*Y^1 \otimes (W^1)^{\perp}) \}.
\end{multline}
Canonical choices are to take $W^1$ equal to the zero section or $W^1$ equal to $\cH^{f^1/2}(H^1/Y^1).$

If the link of $\hat X$ at $Y^2,$ $Z^2,$ is smooth we can carry out the same procedure at $Y^2.$
If $Z^2$ is not smooth, then the compatibility conditions between the strata show that the link of $Z^2$ at its singular stratum is again $Z^1$ (see, e.g., \cite[\S1.2]{ALMP13.1}). Thus the choice of flat bundle $W^1$ induces a domain for the de Rham operator on $Z^2,$
\begin{equation*}
	\cD_{W^1}(\eth_{\dR}^{Z^2})
\end{equation*}
by imposing Cheeger ideal boundary conditions corresponding to $W^1.$
The metric $g$ is suitably scaled at $Y^2$ if at each $Z^2$ the operator $(\eth_{\dR}^{Z_2}, \cD_{W^1}(\eth_{\dR}^{Z^2}))$ does not have non-zero eigenvalues with absolute value less than one.

The null space of $\eth_{\dR}^{Z^2}$ with this domain is finite dimensional and denoted $\cH^*_{W^1}(Z^2).$
These spaces form a bundle over $Y^2,$
\begin{equation*}
	\cH^*_{W^1}(H^2/Y^2) \lra Y^2,
\end{equation*}
which is again naturally endowed with a flat connection and a bundle metric.
We prove in \cite{ALMP13.1} that every differential form $u \in \cD_{\max, W^1}(\eth_{\dR})$ has a distributional asymptotic expansion at $Y^2,$
\begin{equation*}
	u \sim x^{-f^2/2}(\alpha_2(u) + dx \wedge \beta_2(u)) + \wt u'
\end{equation*}
where now $x$ is a boundary defining function for $H^2,$ $f^2$ is the dimension of $Z^2,$ and 
\begin{equation*}
	\alpha_2(u), \beta_2(u) \in H^{-1/2}(Y^2; \Lambda^*T^*Y^2 \otimes \cH^{f^2/2}_{W^1}(H^2/Y^2)), \quad
	\wt u' \in xH^{-1}(X; \Lambda^*\Iie T^*X).
\end{equation*}
(As with $f^1,$ if $f^2$ is odd then our convention is that $\cH^{f^2/2}_{W^1}(H^2/Y^2))=0$ and so $\alpha_2(u)= \beta_2(u)=0.$)
Thus to define Cheeger ideal boundary conditions at $Y^2,$ compatibly with the choice at $Y^1,$ we need to choose a flat subbundle 
\begin{equation*}
	\xymatrix{ W^2 \ar[rr] \ar[rd] & & \cH^{f^2/2}_{W^1}(H^2/Y^2) \ar[ld] \\ & Y^2 & }
\end{equation*}
and then we set
\begin{multline*}
	\cD_{\max, (W^1,W^2)}(\eth_{\dR}) = \{ u \in \cD_{\max, W^1}(\eth_{\dR}) : \\
	\alpha_2(u)  \in H^{-1/2}(Y^2; \Lambda^*T^*Y^2 \otimes W^2), \quad
	\beta_2(u) \in H^{-1/2}(Y^2; \Lambda^*T^*Y^2 \otimes (W^2)^{\perp}) \}.
\end{multline*}
Now proceeding inductively we can give the general definition.

\begin{definition}
Let $(\hat X, g)$ be a stratified space with an $\iie$ metric. A collection of bundles
\begin{equation*}
	\cW = \{ W^1 \lra Y^1, W^2 \lra Y^2, \ldots, W^T \lra Y^T \}
\end{equation*}
is a {\bf (Hodge) mezzoperversity} adapted to $g$ if, inductively, at each $q\in Y^i$ the operator $\eth_{\dR}^{Z^i_q}$ with its induced domain does not have any non-zero eigenvalues of absolute value less than one, and each $W^i$ is a flat subbundle
\begin{equation*}
	\xymatrix{ W^i \ar[rr] \ar[rd] & & \cH^{f^i/2}_{(W^1, \ldots, W^{i-1})}(H^i/Y^i) \ar[ld] \\ & Y^i & }
\end{equation*}
where $\cH^{f^i/2}_{(W^1, \ldots, W^{i-1})}(H^i/Y^i)$ is zero if $f^i$ is odd and otherwise is the bundle over $Y^i$ with typical fiber equal to the null space of $\eth_{\dR}^{Z^i}$ with domain $\cD_{(W^1, \ldots, W^{i-1})}(Z^i).$
\end{definition}
Every Hodge mezzoperversity $\cW$ on $\hat X$ induces a Hodge mezzoperversity on each link $Z$ which we denote $\cW(Z).$

Every mezzoperversity induces a closed self-adjoint domain for $\eth_{\dR},$
\begin{equation*}
	\cD_{\cW}(\eth_{\dR}),
\end{equation*}
which we prove in \cite{ALMP13.1} is compactly included in $L^2(X;\Lambda^*\Iie T^*X).$ Thus $(\eth_{\dR}, \cD_{\cW}(\eth_{\dR}))$ is Fredholm with discrete spectrum, facts already used in the inductive definition of the mezzoperversity. We denote the corresponding Hodge cohomology groups by 
\begin{equation*}
	\cH^*_{\cW}(\hat X)
\end{equation*}
or $\cH^*_{\cW}(\hat X;g)$ when we want to consider the dependence on the metric.\\

We can also use a mezzoperversity to define domain for the exterior derivative as an unbounded operator on $L^2$ differential forms. Namely,
we take the closure of $\cD_{\cW}(\eth_{\dR})$ in the graph norm of $d,$
\begin{equation*}
	\cD_{\cW}(d)
	= \{ \omega \in L^2(X;\Lambda^* \Iie T^*X) : 
	\exists (\omega_n) \subseteq \cD_{\cW}(\eth_{\dR}) \Mst \omega_n \xlra{L^2} \omega \Mand (d\omega_n) \text{ is $L^2$-Cauchy} \}.
\end{equation*}
Similarly we define $\cD_{\cW}(\delta)$ and show that these domains are mutually adjoint and satisfy
\begin{equation}\label{eq:intersec}
	\cD_{\cW}(\eth_{\dR}) = \cD_{\cW}(d) \cap \cD_{\cW}(\delta).
\end{equation}

An important feature of these domains is that, if we denote
\begin{equation*}
	\cD_{\cW}(d)^{[k]} = \cD_{\cW}(d) \cap L^2(X; \Lambda^k (\Iie T^*X))
\end{equation*}
then
\begin{equation*}
	d: \cD_{\cW}(d)^{[k]} \lra \cD_{\cW}(d)^{[k+1]}
\end{equation*}
so that $\cD_{\cW}(d)^{[*]}$ forms a complex. We denote the corresponding de Rham cohomology groups by $\tH^*_{\cW}(\hat X).$ We prove in \cite{ALMP13.1} that
these groups are independent of the choice of metric and that there is a canonical isomorphism
\begin{equation}\label{eq:HodgeDeRham}
	\tH^*_{\cW}(\hat X) \cong \cH^*_{\cW}(\hat X).
\end{equation}

A {\bf de Rham mezzoperversity} is defined like a Hodge mezzoperversity but using de Rham cohomology groups instead of Hodge cohomology groups. Note that we need an $\iie$ metric to talk about a Hodge mezzoperversity, since it involves bundles of harmonic forms. However, as we recall below, the de Rham cohomology groups are metric independent, so we can talk about a de Rham mezzoperversity without reference to a particular metric. Since we can think of a de Rham mezzoperversity as an equivalence class of Hodge mezzoperversities, with the advantage of being metric independent, we often denote a de Rham mezzoperversity by $[\cW].$

In more detail, suppose that $(\hat X, g, \cW)$ and $(\hat X, g', \cW')$ are Hodge mezzoperversities with adapted $\iie$ metrics, say 
\begin{equation*}
	\cW = \{ W^1 \lra Y^1, \ldots, W^T \lra Y^T \}, \quad
	\cW' = \{ (W^1)' \lra Y^1, \ldots, (W^T)' \lra Y^T \}.
\end{equation*}
Since $g$ and $g'$ are quasi-isometric over $\wt X$ as metrics on $\Iie TX,$ we have
\begin{equation*}
	L^2(X; \Lambda^* \Iie T^*X; g) 
	=L^2(X; \Lambda^* \Iie T^*X; g'). 
\end{equation*}
The two Hodge mezzoperversities represent the same de Rham mezzoperversity if, at each $Y^i,$ the Hodge-de Rham isomorphism \eqref{eq:HodgeDeRham} identifies $W^i$ and $(W^i)'.$
We proved in \cite[Theorem 5.9]{ALMP13.1} that in that case
\begin{equation*}
	\cD_{\cW}(d) = \cD_{\cW'}(d), \Mhence \tH^*_{\cW}(\hat X;g) = \tH^*_{\cW'}(\hat X;g') =: \tH^*_{[\cW]}(\hat X)
\end{equation*}
where we use $[\cW]$ to indicate the de Rham mezzoperversity associated to $\cW.$\\

Let us say that an $\iie$ metric $g$ is adapted to the de Rham mezzoperversity $[\cW]$ if the Hodge-de Rham isomorphisms corresponding to $g$ and its induced $\iie$ metrics on the links produce a Hodge mezzoperversity $\cW$ adapted to $g.$

%%%%%%%%%%%%%%%%%%%%%%%%
\section{Twisted de Rham operators and induced domains} \label{sec:TwistedDeRham}
%%%%%%%%%%%%%%%%%%%%%%%%

Let $g$ be an $\iie$ metric on $\hat X$ and let $\eth_{\dR} = d + \delta$ be the corresponding de Rham operator.
If $f \in \CI(\hat X)$ then 
\begin{equation*}
	[\eth_{\dR}, f] 
\end{equation*}
is a linear combination of the exterior product by $df$ and the interior product by $\nabla f,$ and hence is a bounded operator on $L^2(X;\Lambda^* \Iie T^*X).$
It follows that $u \in \cD_{\max}(\eth_{\dR})$ implies $fu \in \cD_{\max}(\eth_{\dR})$ and, since there is a partition of unity consisting of functions in $\CI(\hat X),$ we have
\begin{equation}\label{eq:DmaxIsLocal}
	u \in \cD_{\max}(\eth_{\dR}) \iff fu \in \cD_{\max}(\eth_{\dR}) \Mforall f \in \CI(\hat X).
\end{equation}
We say that a domain for the de Rham operator is {\bf a local domain} if it satisfies \eqref{eq:DmaxIsLocal}.
Directly from their definition, $\cD_{\min}(\eth_{\dR})$ and $\cD_{\cW}(\eth_{\dR})$ are local domains for any mezzoperversity $\cW.$\\

Consider a covering of a stratified space 
\begin{equation*}
	\Gamma - \hat X' \lra \hat X
\end{equation*}
where $\Gamma$ is a countable, finitely generated, finitely presented group, with classifying map $r:\hat X \lra B\Gamma.$
In \cite{ALMP11} we made use of the following fact: 
if $\cU_q \cong \bbR^h \times C(Z)$ is a distinguished neighborhood of a point $q \in \hat X,$ then the induced $\Gamma$-coverings over $\cU_q$ and over $C(Z)$ are trivial (since the cone induces a nullhomotopy).
It follows that if we have a representation of $\Gamma$ on $\bbR^N$  and we form the flat bundle $E = \hat X' \btimes_{\Gamma} \bbR^N$ and the twisted de Rham operator $\eth_{\dR}^E$ then over a distinguished neighborhood we can identify
\begin{equation*}
	E\rest{\cU_q} \cong \cU_q \times \bbR^N, \quad \eth_{\dR}^E\rest{\cU_q} = \eth_{\dR}\rest{\cU_q} \otimes \Id_{\bbR^N}
\end{equation*}
since not only the bundle $E$ is trivial over $\cU_q$ but also its connection. (In this last equality and below we abuse notation, as we should more correctly conjugate $\eth_{\dR}^E$ by the identification of $E\rest{\cU_q}$ with $\cU_q \times \bbR^N$ to obtain $\eth_{\dR}\rest{\cU_q} \otimes \Id_{\bbR^N}.$)

We can use this to define a domain for $\eth_{\dR}^E$ for each mezzoperversity $\cW.$
Indeed, on any distinguished neighborhood we have $L^2(\cU_q;E) = L^2(\cU_q)^N,$ so it makes sense to define
\begin{multline*}
	\cD_{\cW}(\eth_{\dR}^E) = \{ u \in L^2(X; \Lambda^*\Iie T^*X \otimes E): 
	\eth_{\dR}^Eu \in L^2(X; \Lambda^*\Iie T^*X \otimes E) \text{ and, } \\
	\text{ for all } f \in \CI(\hat X) \text{ supported in a distinguished neighborhood, }
	fu \in \cD_{\cW}(\eth_{\dR})^N \}.
\end{multline*}
Notice that, because $\cD_{\cW}(\eth_{\dR})$ is local, we have $\cD_{\cW}(\eth_{\dR}^E) = \cD_{\cW}(\eth_{\dR})$ when $E$ is the trivial flat $\bbR$-bundle.
On the other hand each domain $\cD_{\cW}(\eth_{\dR}^E)$ is local by construction.

In the same way, let us denote the flat connection on $E$ by $d^E$ and define
\begin{multline*}
	\cD_{\cW}(d^E) = \{ u \in L^2(X; \Lambda^*\Iie T^*X \otimes E): 
	d^Eu \in L^2(X; \Lambda^*\Iie T^*X \otimes E) \Mand, \\
	\text{ for all } f \in \CI(\hat X) \text{ supported in a distinguished neighborhood, }
	fu \in \cD_{\cW}(d)^N \}.
\end{multline*}
Since, by construction, these domains are just copies of the untwisted domains near each singular stratum, Theorem 5.6 in \cite{ALMP13.1} implies the following:

\begin{theorem}
Let $(X,g,\cW)$ be a stratified space endowed with an $\iie$-metric $g$ and adapted Hodge mezzoperversity. 
Let $\Gamma$ be a countable, finitely generated, finitely presented group, $\hat X' \lra \hat X$ a $\Gamma$ covering, and $E \lra X$ a flat bundle associated to an orthogonal representation of $\Gamma$ on $\bbR^N.$
The operator $(\eth_{\dR}^E, \cD_{\cW}(\eth_{\dR}^E))$ is self-adjoint and its domain is compactly included in $L^2(X; \Lambda^*\Iie T^*X \otimes E),$ hence $\eth_{\dR}^E$ is Fredholm with compact resolvent. In particular the corresponding Hodge cohomology groups
\begin{equation*}
	\cH^*_{\cW}(\hat X;E)
\end{equation*}
are finite dimensional.

The operators $(d^E, \cD_{\cW}(d^E))$ define a Fredholm Hilbert complex and hence finite dimensional de Rham cohomology groups
\begin{equation*}
	\tH^*_{\cW}(\hat X;E),
\end{equation*}
which are canonically isomorphic to the Hodge cohomology groups.
The domains $\cD_{\cW}(d^E),$ and hence these groups, are independent of the choice of metric and depend on $\cW$ only through the de Rham mezzoperversity $[\cW].$
\end{theorem}

A key fact for this paper is that all of the above is true if we replace the finite rank bundle $E \lra \hat X$ with certain bundles of projective finitely generated modules over a $C^*$-algebra.
Let us assume that we have a $\Gamma$-covering as above, but now consider the action of $\Gamma$ on $C^*_r\Gamma,$ the reduced $C^*$-algebra of $\Gamma$ (i.e., the closure of the algebra of operators generated by the elements of $\Gamma$ in the left regular representation),
to define a bundle $\sG(r)$ of free left $C^*_r\Gamma$-modules of rank one:
\begin{equation}\label{sGDef}
	\sG(r) = C^*_r\Gamma \times_\Gamma \hat X'
\end{equation} 
The vector space of smooth sections with compact support of the bundle $\Lambda^*\Iie T^*X \otimes \sG(r) 
\rightarrow X$ is endowed with the (usual) $C^*_r \Gamma-$scalar product defined by:
$$
\langle \sum_{\gamma_1} a_{\gamma_1} \gamma_1 ;  \sum_{\gamma_2} b_{\gamma_2} \gamma_2 \rangle = 
\sum (a_{\gamma_1} ; \gamma_1 \cdot b_{(\gamma_2)^{-1}})_{L^2} \,\, \gamma_1 (\gamma_2)^{-1} \in C^*_r \Gamma\,.
$$ Its completion with respect to $\langle \cdot ; \cdot \rangle$ is denoted $L^2(X; \Lambda^*\Iie T^*X \otimes \sG(r)).$ (Note that this notation is slightly different from that employed in \cite{ALMP11}.)

The trivial connection on $C^*_r\Gamma \times \hat X'$ induces a flat connection on $\sG(r),$ and we denote the corresponding twisted exterior derivative (i.e., the flat connection) by $d^{\sG(r)}$ and the twisted de Rham operator by $\eth_{\dR}^{\sG(r)}.$
As above, we can use localization to define a domain $\cD_{\cW}(\eth_{\dR}^{\sG(r)})$ for every mezzoperversity $\cW.$ 
For example over a distinguished neighborhood we can identify
\begin{equation*}
	\sG(r) \rest{\cU_q} \cong \cU_q \times C^*_r\Gamma, \quad \eth_{\dR}^{\sG(r)}\rest{\cU_q} = \eth_{\dR}\rest{\cU_q} \otimes \Id_{C^*_r\Gamma}
\end{equation*}
Then the previous theorem implies (cf. \cite[Proposition 6.3]{ALMP11}):

\begin{theorem}\label{thm:HigherSA}
With the notation of the previous theorem,
the operator $(\eth_{\dR}^{\sG(r)}, \cD_{\cW}(\eth_{\dR}^{\sG(r)}))$ is self-adjoint and its domain is $C^*_r \Gamma-$compactly included in $L^2(X; \Lambda^*\Iie T^*X \otimes \sG(r) );$ in particular it has a $C^*_r \Gamma-$compact resolvent. 
\end{theorem}

%%%%%%%%%%%%%%%%%%%%%%%%
\section{Stratified homotopy invariance of twisted de Rham cohomology} \label{HomotopyInv} 
%%%%%%%%%%%%%%%%%%%%%%%%

We define morphisms between iterated fibration structures and show that they induce pull-back maps on $L^2$-cohomology.
Iterated edge metrics are used both to define the cohomologies and to define the pull-back map.
We show that a (smooth codimension preserving) stratified homotopy equivalence induces an equivalence in $L^2$-cohomology. For stratified diffeomorphisms, this map is just pull-back of differential forms, but for a general homotopy equivalence we use the `Hilsum-Skandalis replacement' construction from \cite{ALMP11, Hilsum-Skandalis}.

%%%%%%%%%%%%%%%%%%%%%%%%%%%%%
\subsection{Gauss-Manin connection}\label{sec:GMConn}
%%%%%%%%%%%%%%%%%%%%%%%%%%%%%

It will be useful to have a more invariant description of the flat connections on the vertical cohomology bundles of the links, so we start by describing this.\\

Let $(\hat X,g,\cW)$ be a stratified space equipped with an $\iie$ metric and a (de Rham) mezzoperversity $\cW.$
Each hypersurface $H^i$ of $\wt X,$ the resolution of $\hat X,$ is the total space of a fibration
\begin{equation*}
	H^i \xlra{\phi_i} Y^i
\end{equation*}
with fiber $\wt Z^i,$ the resolution of the link of $\hat X$ at $Y^i.$

For any $\ell \in \bbN_0,$ the vertical cohomology of $L^2$-de Rham cohomology with Cheeger ideal boundary conditions forms a bundle
\begin{equation*}
	\tH^\ell_{\cW(Z^i)}(Z^i) -
	\tH^{\ell}_{\cW(Z^i)}(H^i/Y^i) \lra Y^i
\end{equation*}
and we described in \cite{ALMP13.1} a flat connection for this bundle using a connection for $\phi_i.$

This is a version of the Gauss-Manin connection for $\phi_i$ which, as is well-known \cite{Katz-Oda},  can be obtained through the Leray-Serre spectral sequence without using a connection for $\phi_i.$ We now review this construction in the $L^2$-setting.

\begin{lemma}
There is a natural {\bf $L^2$-Gauss-Manin connection} $\nabla^{\tH}$ on the vertical cohomology bundle 
\begin{equation*}
	\tH^{\ell}_{\cW(Z^i)}(H^i/Y^i) \lra Y^i
\end{equation*}
that does not depend on the choice of $\iie$-metric or on a choice of connection for $\phi_i.$
\end{lemma}

\begin{proof}
We describe the smooth construction in parallel to the $L^2$ construction as the former makes the latter more transparent.
Recall that we have shown that the cohomology groups $\tH^{\ell}_{\cW(Z^i)}(H^i/Y^i)$ only depend on the de Rham mezzoperversity $\cW(Z^i)$ and not on the $\iie$ metric.

Recall that the $\iie-$geometry is defined in Section 1. We then have a short exact sequences of bundles over $H^i,$ %
\begin{equation}\label{eq:BundleSES}
\begin{gathered}
	0 \lra \Iie TH^i/Y^i \lra \Iie TH^i \lra \phi^*(\Iie TY^i)  \lra 0 \\
	0 \lra \phi^*(\Iie T^*Y^i) \lra \Iie T^*H^i \lra \Iie T^*(H^i/Y^i) \lra 0
\end{gathered}
\end{equation}
which induce short exact sequences of sections over $H^i,$ e.g., 
\begin{equation*}
\begin{gathered}
	0 \lra \CI(H^i;\Lambda^* \phi^*(\Iie T^*Y^i)) \lra \CI(H^i; \Lambda^* \Iie T^*H^i) \xlra{\pi_{\infty}} \CI(H^i; \Lambda^* \Iie T^*(H^i/Y^i)) \lra 0 \\
	0 \lra L^2(H^i;\Lambda^* \phi^*(\Iie T^*Y^i)) \lra L^2(H^i; \Lambda^* \Iie T^*H^i) \xlra{\pi_{L^2}} L^2(H^i; \Lambda^* \Iie T^*(H^i/Y^i)) \lra 0.
\end{gathered}
\end{equation*}
For our purposes we need the corresponding sequence for $L^2$-differential forms in the domain of $d,$
\begin{equation}\label{eq:DomainSES}
	0 \lra \cD_{\cW(Z^i)}(d_{H^i}) \cap L^2( H^i; \Lambda^*(\phi^*\Iie TY^i) )
	\lra \cD_{\cW(Z^i)}(d_{H^i})
	\lra \pi_{L^2}(\cD_{\cW(Z^i)}(d_{H^i})) \lra 0.
\end{equation}
It is useful to observe that each of 
\begin{equation*}
	\CI(H^i; \Lambda^*(\Iie T^*H^i) ), \quad L^2(H^i; \Lambda^*( \Iie T^*H^i ) ), \quad \cD_{\cW(Z^i)}(d_{H^i})
\end{equation*}
is a module over $\phi^*\CIc(Y^i).$

Note that since we are not imposing boundary conditions on the boundary of $Y^i,$ the first non-zero term in \eqref{eq:DomainSES} can be identified with
\begin{equation*}
	\{ \omega \in L^2( H^i; \Lambda^*(\phi^*\Iie TY^i) ) : d\omega \in L^2( H^i; \Lambda^*(\phi^*\Iie TY^i) ) \},
\end{equation*}
while the last non-zero term can be thought of as forms whose restriction to each fiber is in $\cD_{\cW}(Z^i)(d_{Z_i}).$ We can make this precise by using the invariance under multiplication by $\CIc(Y^i)$ to localize to a neighborhood and mollify in $Y^i$ and thereby obtain a dense subdomain of forms which can be restricted to an individual fiber.

The exact sequence \eqref{eq:BundleSES} induces Cartan's filtration
\begin{multline*}
	 F_j^\ell \CI 
	= \CI( H^i; \Lambda^j(\phi^*(\Iie T^*Y^i)) ) \;\hat\otimes\; \CI( H^i;\Lambda^{\ell - j}(\Iie T^*H^i) )\\
	= \text{ degree $\ell$ differential forms on $H^i$ with `horizontal degree' at least } j
\end{multline*}
which can also be described by
\begin{equation*}
	F_j^{\ell}\CI = \{\omega \in \CI( H^i;\Lambda^{\ell}(\Iie T^*H^i) ) : \omega(V_1, \ldots, V_\ell)=0 \Mwhenever \ell-j+1 \text{ of the $V_k$ are vertical } \},
\end{equation*}
where a vector field is vertical when it is a section of $TH^i/Y^i.$
We similarly obtain filtrations of $L^2$ differential forms
\begin{equation*}
	F_j^{\ell}L^2 = \{\omega \in L^2( H^i;\Lambda^{\ell}(\Iie T^*H^i) ) : \omega(V_1, \ldots, V_\ell)=0 \Mwhenever \ell-j+1 \text{ of the $V_k$ are vertical } \}.
\end{equation*}
and $F_{j, \cW(Z^i)}^{\ell}= F_j^\ell L^2 \cap \cD_{\cW(Z^i)}(d).$
These filtrations are compatible with the exterior derivative
\begin{equation*}
	d: F_j^*\CI \lra F_j^*\CI, \quad d:F_j^* L^2 \cap \cD_{\max}(d) \lra F_j^*L^2, \quad d: F_j^*\cD_{\cW(Z^i)} \lra F_j^*\cD_{\cW(Z^i)}
\end{equation*}
so there are associated spectral sequences.
As in \cite{Katz-Oda} we show that the differential on the $E_1$ page of the spectral sequence yields the connection we are looking for.

We can use  \eqref{eq:BundleSES} to compute $E_0,$ the graded complex associated to the filtration.
First note that
\begin{multline*}
	E_0^{j,k}\CI = F_j^{j+k}\CI / F_{j+1}^{j+k}\CI \\
	= \CI(H^i; \Lambda^j(\phi^*T^*Y^i)) \;\hat\otimes\; \CI(H^i;\Lambda^{k}T^*(H^i/Y^i))
	= \CI \lrpar{ H^i; \Lambda^j(\phi^*T^*Y^i) \;\hat\otimes\; \Lambda^{k}T^*(H^i/Y^i) }
\end{multline*}
inherits a differential 
\begin{equation*}
	\pa_{0,\CI}^{j,k}: E_0^{j,k}\CI \lra E_0^{j,k+1}\CI,
\end{equation*}
that only raises the `vertical degree'.
Moreover since the individual terms in the Cartan filtration are $\phi^*\CIc(Y^i)$-modules, this differential is linear over $\phi^*\CIc(Y^i).$
This allows us to localize to a contractible neighborhood of $Y^i$ and identify $\pa_{0,\CI}^{j,k}$ with the vertical exterior derivative,
\begin{equation*}
	\pa_{0,\CI}^{j,k} = d_{H^i/Y^i}.
\end{equation*}
Similarly we can identify
\begin{equation*}
	E_0^{j,k}L^2 
	= L^2(H^i; \Lambda^j(\phi^*T^*Y^i) \;\hat\otimes\; \Lambda^{k}T^*(H^i/Y^i) )
\end{equation*}
and
\begin{equation*}
	\pa_{0,L^2}^{j,k} = d_{H^i/Y^i}
\end{equation*}
is the densely defined vertical exterior derivative.
The same holds, {\em mutatis mutandis}, on $E_0^{j,k}\cD_{\cW(Z^i)} = F_j^{j+k}\cD_{\cW(Z^i)} / F_{j+1}^{j+k}\cD_{\cW(Z^i)}.$

The $E_1$ page is the cohomology of the $E_0$ page with the induced differential, so by the description of the latter we have
\begin{equation*}
	E_1^{j,k}\CI = \frac{\ker \pa_{0,\CI}^{j,k}}{\Image \pa_{0,\CI}^{j,k-1}} = 
	\CI( Y^i; \Lambda^j T^*Y^i \;\hat\otimes\; \tH^k_{\CI}(H^i/Y^i) )
	=: \Omega^j( Y^i; \tH^k_{\CI}(H^i/Y^i) )
\end{equation*}
Similarly
\begin{equation*}
\begin{gathered}
	E_1^{j,k}\cD_{\cW(Z^i)} = 
	\{ \omega \in L^2(Y^i; \Lambda^j T^*Y^i \;\hat\otimes\; \tH^k_{\cW(Z^i)}(H^i/Y^i) ) : d\omega \in L^2(Y^i; \Lambda^j T^*Y^i \;\hat\otimes\; \tH^k(H^i/Y^i) ) \} \\
	=: \Omega^j_{L^2}( Y^i; \tH^k_{\cW(Z^i)}(H^i/Y^i) )
%	= \cD_{\max}^{[j]}\lrpar{ d_{Y^i; \tH^k_{\cW(Z^i)}(H^i/Y^i)} }
\end{gathered}
\end{equation*}
The exterior derivative on $H^i$ induces a differential on this page 
\begin{equation}\label{eq:E1diffs}
\begin{gathered}
	\pa_{1, \CI}^{j,k}: 
	\Omega^j( Y^i; \tH^k_{\CI}(H^i/Y^i) ) \lra
	\Omega^{j+1}( Y^i; \tH^k_{\CI}(H^i/Y^i) ) \\
	\pa_{1, L^2}^{j,k}: 
	\Omega^j_{L^2}( Y^i; \tH^k_{\cW(Z^i)}(H^i/Y^i) ) \lra
	\Omega^{j+1}_{L^2}( Y^i; \tH^k_{\cW(Z^i)}(H^i/Y^i) )
\end{gathered}
\end{equation}
by restriction from $\CI(H^i; \Lambda^* \Iie T^*H^i)$ to the kernel of the composition 
\begin{equation*}
	F_j^{j+k}\CI \lra F_j^{j+k+1}\CI \lra F_j^{j+k+1}\CI/ F_{j+1}^{j+k+1}\CI
\end{equation*}
and similarly for $L^2$-forms.

Note that 
$\CIc(Y^i; \Lambda^*T^*Y^i) = \Omega_c^*(Y^i)$ 
is a subset of both $\Omega^*( Y^i; \tH^0_{\CI}(H^i/Y^i) )$ and, because of the compact support, also of
$\Omega^*_{L^2}( Y^i; \tH^0_{\cW(Z^i)}(H^i/Y^i) ).$
Both of the $E_1$ pages are closed under exterior product with a form in $\Omega_c^j(Y^i)$ and satisfy
\begin{equation*}
	\omega \in \Omega_c^{j'}(Y^i), \eta \in E_1^{j,k} \implies 
	\pa_1^{j+j',k}(\omega \wedge \eta) = d\omega \wedge \eta +(-1)^{j'} \omega \wedge \pa_1^{j,k}\eta
\end{equation*}
with their respective differentials.
However this shows that 
\begin{equation*}
\begin{gathered}
	\pa_{1,\CI}^{0,k}: \CI(Y^i; \tH^k_{\CI}(H^i/Y^i)) \lra  \Omega^1( Y^i; \tH^k_{\CI}(H^i/Y^i) ) \\
	\pa_{1,L^2}^{0,k}: \Omega^0_{L^2}(Y^i; \tH^k_{\cW(Z^i)}(H^i/Y^i)) \lra  \Omega^1_{L^2}( Y^i; \tH^k_{\cW(Z^i)}(H^i/Y^i) ) 
\end{gathered}
\end{equation*}
are connections and that the maps \eqref{eq:E1diffs} are induced by these connections.
In particular $\pa_{1}^{1,k} \circ \pa_{1}^{0,k}$ is exterior multiplication by the curvature of the connection, and  since $\pa_1$ is a differential, the connections are flat.
We denote these connections by $\nabla^{\tH}$ in either case.

Finally note that the connection on $\tH^k_{\cW(Z^i)}(H^i/Y^i)$ is independent of choices since all rigid, suitably scaled, $\iie$ metrics yield the same $L^2$ space and the same $\cD_{\cW(Z^i)}(d_{H^i}).$
\end{proof}

\begin{remark}\label{rmk:GaussManin}
It follows from the proof of this lemma that a map on differential forms between the total spaces of two fibrations will induce maps intertwining the Gauss-Manin connection if:\\
i) it commutes with the exterior derivative, and\\
ii) it is compatible with Cartan's filtration, i.e., it takes degree $k$ forms pulled-back from the base to degree $k$ forms pulled-back from the base.
\end{remark}

%%%%%%%%%%%%%%%%%%%%%%%%%%%%%
\subsection{Iterated fibration morphisms}\label{sec:IfsHom}
%%%%%%%%%%%%%%%%%%%%%%%%%%%%%

Given a smooth map $f:\wt X \lra \wt M$ between manifolds with corners and iterated fibration structures, we can pull-back differential forms on $\wt M$ to differential forms on $\wt X.$ However this will not pull-back $\ie$-forms to $\ie$-forms (nor $\iie$-forms to $\iie$-forms) unless $f$ is a $b$-map whose restriction to each boundary hypersurface is a fiber bundle map. For this reason we think of these maps as the `morphisms' in this setting.

Recall from Definition \ref{def:SmoothBFun} that $F \in \CI_b(\hat X, \hat M)$ means that $F \in \cC^0(\hat X, \hat M)$ has a lift
\begin{equation*}
	\wt F: \wt X \lra \wt M
\end{equation*}
which is a smooth $b$-map between manifolds with corners; note that $\wt F$ necessarily preserves the boundary fibration structures in that its restriction to a boundary hypersurface is a fibre bundle map.
The differential of a smooth $b$-map extends by continuity from the interior to a map between the iterated edge tangent bundles 
\begin{equation*}
	D\wt F: {}^{\ie}TX \lra {}^{\ie}TM
\end{equation*}
and we say that $F$ is an {\bf $\ie$-submersion} if this map is surjective.
 %{\em when restricted to the interior}.\red{This needs changing ....}
As an example, consider a boundary hypersurface $H$ in a manifold with corners $\wt M$ endowed with an iterated fibration 
structure. By definition, there is a fibration $H \xlra \phi \wt Y$, and both $H$ and $\wt Y$ inherit boundary fibration structures 
from $\wt M$.  It follows, directly from the definitions, that $\phi$ is a smooth $b$-map.

The exponential mapping associated to an $\ie$ metric $g$ provides another important example of an $\ie$-submersion.
We explain this carefully. First note that the unit ball bundle ${}^{\ie}\bbB M \subset {}^{\ie}TM$ is an open manifold with an 
iterated fibration structure described as follows. To each boundary hypersurface $H$ of $\wt M$ corresponds a boundary hypersurface
\begin{equation*}
	\Ie \bbB_HM = \Ie\bbB M\rest{H}
\end{equation*}
in $\Ie \bbB M$, which is is the total space of a fibration $\psi:  \Ie \bbB_HM \lra \wt Y$, where $\psi$ is the composition of the maps
\begin{equation*}
	\Ie \bbB_HM \lra \Ie\bbB H \xlra{D\phi} \Ie \bbB Y \lra \wt Y
\end{equation*}
(the first map here is orthogonal projection and $D\phi$ is the differential of the fibration $\phi:H \lra \wt Y$).  

\begin{lemma} The exponential map $\exp: \Ie \bbB M\rest{M} \lra M$ 
extends uniquely to a smooth $\ie$-submersion 
\begin{equation*}
	\bar\exp: \Ie\bbB M \lra \wt M,
\end{equation*}
which for each boundary hypersurface $H$ fits into the commutative diagram 
\begin{equation}\label{eq:expcoversid}
	\xymatrix{
	\Ie \bbB_HM \ar[rr]^-{\bar\exp\rest{H}}  \ar[d]^{\psi} & & H \ar[d]^{\phi} \\
	\wt Y \ar[rr]^-{\id} & & \wt Y}
\end{equation}
\end{lemma} 
\begin{proof} We first prove this when $M$ has a simple edge, so that the tubular neighborhood of its singular stratum
is a bundle of cones, each of which has a smooth compact cross-section. In this case, the resolution $\wt M$ is a manifold 
with fibered boundary:
$$ \phi: \partial \wt M \rightarrow  Y\,,$$ 
with typical fiber $Z$.  Introduce local coordinates $(x,y,z)$ so that 
$$
g = \frac{ dx^2}{x^2} + \frac{\phi^* g_Y}{x^2} + h, 
$$ 
where $h$ is a symmetric $2$-tensor orthogonal to $dx/x$ which restricts to a smooth family of metrics on each fiber 
$Z \simeq \phi^{-1} \{y\}$.  
%The orthogonal complement $H_p$ to the vertical tangent space at any point $p$ is spanned by the vectors $\partial_{y_j} + 
%A_p (\partial_{y_j})$, where $A_p: T_{y} Y \rightarrow T_z Z$, $p = (x,y,z)$. 

Now let $\gamma(t) = (x(t), y(t), z(t))$ denote a geodesic starting at some point $p_0 = (x_0, y_0, z_0)$ with $x_0 >0$ and with
initial tangent vector $(\dot{x}_0, \dot{y}_0, \dot{z}_0) \in T_{p_0} M$ with $g$-norm no greater than $1$. We analyze the 
behavior of the first two coordinates $x(t)$ and $y(t)$ as the parameter $x_0 \searrow 0$. 
% Observe that
% $$
% (y'(t), z'(t)) = (y'(t), A_{z(t)} y'(t) )+ (0, z'(t)- A_{z(t)} y'(t)) \in T_{z(t)}Z\oplus T_{\gamma(t)}Z^\perp,
% $$ 
% and clearly, 
% $$
% g(y'(t), A_{z(t)} y'(t) )= \frac{g_Y( y'(t))}{x(t)^2}\, .
% $$
Indeed, from $|\gamma'(t)|_g\leq 1$ we obtain 
\[
\left|\frac{x'(t)}{x(t)}\right| \leq C, \quad g_Y(y'(t), y'(t)) \leq C x(t)^2,
\]
whence $x(t) = x_0 B_1(t)$, where $|B_1(t)| \leq C$ independently of initial conditions and for $|t| \leq \epsilon_0$.  Then from
$|y'(t)| \leq x(t)B_2(t)$, where $B_2(t)$ is similarly bounded, we obtain finally that $y(t) = y_0 + x_0 B_3(t)$. 
% $$
% + \frac{g_Y( y'(t))}{x(t)^2} + h_{y(t)} ( z'(t)-A_{z(t)} y'(t))\, \leq 1, \quad 0 \leq t < \epsilon_0.
% $$ 
% We can then write 
% $$
% \frac{x'(t)}{x(t)} = \rho(t) \cos \phi(t) ,\ \mbox{and}\ \  \frac{g_Y( y'(t))}{x(t)^2} = \rho(t)^2 \sin^2 \phi(t), \, 
% $$ 
% where $\rho(t) \in [0,1]$. One gets 
% $$
% x(t)= x_0 \exp { \int_0^t \rho(u) \cos \phi(u)\, du }\,. 
% $$ 
% Then, since $g_Y( y'(t))= x(t)^2 \rho(t)^2 \sin^2 \phi(t)$, one  gets: $y'(t)= x_0 B(t)$ where the norm of $B(t)$ remains bounded independently on 
% $x_0, y_0, t\in [0,\epsilon]$. Then, one gets:
% $$
% y(t)= y_0 + x_0 C(t) \, ,
% $$ where the norm of $C(t)$ remains bounded independently on 
% $x_0, y_0, t$.
Now letting $x_0 \to 0$, one obtains that $x(t)= 0, y(t)= y_0$ for all $t\in [0,\epsilon]$ for all geodesics starting at a point
$(0, y_0, z_0)$.  This proves the result for depth one spaces.

Next assume that $M$ has depth two. Using the notation above, replace $Y$ by $Y_0$ and assume that $Z$ is a bundle of cones 
$C(Z_1)$ over the smooth manifold $Z_1$. In local coordinates $(x_0, y_0,x_1, y_1, z_1)$, the complete metric $g$ takes the form
$$
\frac{ d x_0^2}{ x_0^2 x_1^2} + \frac{ \phi_0^* g_{Y_0}}{x_0^2 x_1^2} + \frac{d x_1^2}{ x_1^2} + \frac{\phi_1^* g_{Y_1}}{ x_1^2} + h_{Z_1}\,.
$$ 
A geodesic $\gamma(t) = (x_0(t), y_0(t), x_1(t), y_1(t), z_1(t))$ with $|\gamma'(t)| \leq 1$ satisfies
\[
\left| \frac{x_0'(t)}{x_0(t)x_1(t)}\right| \leq C, \quad g_{Y_0}(y_0'(t), y_0'(t)) \leq C (x_0(t) x_1(t))^2,\quad 
\left| \frac{x_1'(t)}{x_1(t)}\right| \leq C, \quad g_{Y_1}(y_1', y_1') \leq C x_1(t)^2.
\]
Arguing as before, 
\begin{multline*}
x_1(t) = x_1(0) B_1(t), \quad x_0(t) = x_0(0)x_1(0) B_0(t),  \\ y_1(t) = y_1(0) + x_1(0) \tilde{B}_(t), \quad
y_0(t) = y_0(0) + x_0(0)x_1(0)\tilde{B}_0(t), 
\end{multline*}
where the coefficient functions are all uniformly bounded and smooth.  This implies the corresponding conclusion in this case as well.

The general case is very similar and is left to the reader. 

 \end{proof}

It follows that $\bar\exp_M$ restricts to a map of each fiber of $\phi,$ and indeed it is easy to see that this is
the exponential map for the induced $\ie$-metric on the fiber, i.e., 
\begin{equation*}
	\xymatrix{
	\Ie \bbB_HM\rest{\phi^{-1}(q)} \ar[rr]^{\bar\exp_M} \ar[rd]_{\pi} &  & \phi^{-1}(q)\\
	& \Ie\bbB \phi^{-1}(q) \ar[ru]_{\bar\exp_{\phi^{-1}(q)}} & }
\end{equation*}
where $\pi$ is the orthogonal projection.
\subsection{Hilsum-Skandalis maps} \label{sec:HSmaps}
%%%%%%%%%%%%%%%%%%%%%%%%%%%%%

If $\hat X$ and $\hat M$ are smoothly stratified spaces and $F \in \CI_b(\hat X,\hat M)$ then pull-back of forms defines a map
\begin{equation*}
	F^*: \CI(M; {}^{\iie}\Lambda^*) \lra \CI(X; {}^{\iie}\Lambda^*)
\end{equation*}
which unfortunately need not act continuously on $L^2.$
However we explained in \cite[\S9.3]{ALMP11}, following \cite{Hilsum-Skandalis}, how to replace $F^*$ with a `Hilsum-Skandalis map'
\begin{equation*}
	HS(F): \CI(M; {}^{\iie}\Lambda^*) \lra \CI(X; {}^{\iie}\Lambda^*)
\end{equation*}
which extends to a bounded operator on $L^2.$
Recall that this map is defined on ${}^{\ie}\Omega^*$ (and then on ${}^{\iie}\Omega^*$) using an $\ie$-metric $g_M$ on $M$ 
and a Thom form $\cT_M$ for ${}^{\ie}\bbB M \lra M,$ by:
 first defining $\pi_X$ and $\bbB(F)$ through the natural diagram
\begin{equation*}
	\xymatrix{
	F^*({}^{\ie}\bbB M) \ar[rr]^-{\bbB(F)} \ar[d]^{\pi_X} & & {}^{\ie}\bbB M \ar[d]^{\pi_M} \\
	X \ar[rr]^-{F} & & M }
\end{equation*}
and then setting
\begin{equation*}
	HS(F)(\omega) 
	= (\pi_X)_*( \bbB(F)^*\cT_M \wedge (\bar\exp_M \circ \bbB(F) )^*\omega)
	\Mforall \omega \in \CI(M; {}^{\ie}\Lambda^*).
\end{equation*}

We will make use of a boundary version of this map.
Let $Y_X$ be a singular stratum of $X$ and let $F$ map $Y_X$ into $Y_M$ a singular stratum of $M.$
With the notation above, the lift of $F$ acting between the resolutions of $X$ and $M$ participates in
\begin{equation}\label{eq:wtFcoversF}
	\xymatrix{
	H_X \ar[d]^{\phi_X} \ar[r]^{\wt F} & H_M \ar[d]^{\phi_M} \\
	Y_X \ar[r]^F & Y_M }
\end{equation}
We let $J: \Ie \bbB H_M \hookrightarrow \Ie \bbB_{H_M}M$ be the inclusion and define
\begin{equation*}
	HS_{Y_X}(\wt F)(\omega) 
	= (\pi_X\rest{H_X})_*( \bbB(\wt F)^*\cT_M\rest{H_M} \wedge (\bar\exp_M\rest{H_M} \circ J \circ \bbB(\wt F) )^*\omega)
	\Mforall \omega \in \CI(H_M; {}^{\ie}\Lambda^*  T^*H_M).
\end{equation*}
The fact that the exponential map on $M$ limits to the exponential map on the fibers of $\phi_M$ indicates that the map $HS(F)$ should limit to $HS_{Y_X}(F)$ which we now make precise.

\begin{lemma}\label{lem:RegPullBack}
Given $q\in Y_X,$ let $\cU_{F(q)} \cong [0,1)_{x_M} \times \bbB \times Z^M_{F(q)}$ be a distinguished neighborhood of $F(q) \in Y_M,$ and let $\omega \in \cD_{\cW}(d_M),$ supported in $\cU_{F(q)},$ have the form
\begin{equation*}
	\omega = \phi x_M^{-f/2}(\alpha(\omega)+dx_M\wedge\beta(\omega)) + \wt \omega
\end{equation*}
where $\phi \in \CI(\hat M)$ is a smooth cutoff function, $\alpha$ and $\beta$ are smooth in $\bbB$ and independent of $x_M$ (as sections of $\Iie T^*M$), and $\wt\omega$ is in $x_ML^2(M;\Lambda^* \Iie T^*M).$
Note that $\alpha(\omega)$ and $\beta(\omega)$ both necessarily have vertical degree $f/2.$

Then $\eta = HS(F)\omega$ has an expansion at $Y_X$ with leading term
\begin{equation*}
	x_X^{-f/2}(\alpha(\eta) + dx_X \wedge \beta(\eta))
\end{equation*}
and we have
\begin{equation}\label{eq:BdyHSAction}
	x_X^{-f/2}\alpha(\eta)\rest{H_X} = HS_{Y_X}(\wt F)( x_M^{-f/2}\alpha(\omega)\rest{H_M} ).
\end{equation}

\end{lemma}

\begin{proof}
Note first that we can ignore the `error term' $\wt\omega;$ indeed, since $\wt F$ is a $b$-map, the pull-back $x_M$ under $\wt F$ is
a product of nonnegative integer powers of defining functions, one of is certainly $x_X$. Thus $HS(F)\wt \omega \in 
x_X L^2(X;\Lambda^* \Iie T^*X).$

If $\psi$ is a smooth function on $\wt M$ that is constant on the fibers of the boundary fibration, then $HS(F)(\psi)$ is a smooth function on $\wt X$ that is constant on the fibers of the boundary fibrations. Indeed, the value of $HS(F)(\psi)$ at a point on the fiber over $q \in Y_X$ is an average of the values of $\psi$ on the fiber over $F(q) \in Y_M,$ but $\psi$ is constant on the fibers.

In particular, $HS(F)(\phi)$ will be a smooth cut-off function on $X$ with support in a distinguished neighborhood $\cU_q$ of $q,$ and independent of $x_X$ in a small enough distinguished neighborhood of $q.$

Next note that since $\alpha(\omega)$ is independent of $x_M$ as a section of $\Iie T^*M,$ $x_M^{-f/2}\alpha(\omega)\rest{H_M}$ is naturally a section of $\Lambda^*\Iie T^*H_M,$ and so the right hand side of \eqref{eq:BdyHSAction} makes sense.
In fact it is convenient to interpret both 
\begin{equation*}
	\wt\alpha(\omega) = x_M^{-f/2}\alpha(\omega)
	\Mand
	\wt\beta(\omega) = x_M^{-f/2}\beta(\omega)
\end{equation*}
as sections of $\Lambda^*\Iie T^*H_M$ pulled-back to $\cU_{F(q)},$ which we denote $j^*\Lambda^*\Iie T^*H_M.$

The function
\begin{equation*}
	p = \bar \exp_M \circ \bbB(F)
\end{equation*}
is smooth down to $\{x_X=0\}$ where it restricts to
\begin{equation*}
	p_0 = \bar\exp_M\rest{H_X} \circ \bbB(\wt F)\rest{\pi_X^{-1}H_X}.
\end{equation*}
Given a section $\gamma$ of $j^*\Lambda^*\Iie T^*H_M,$ we can pull it back by either $p$ or $p_0,$ and clearly we have
\begin{equation*}
	p^*\gamma = p_0^*\gamma + dx \wedge \gamma' + \gamma'', \Mforsome \gamma', \gamma'' \Mwith |\gamma''| \leq C x_X.
\end{equation*}
This shows that $p^*(x^{-f/2}\alpha(\omega))$ will have vertical degree $f/2$ at $x_X=0,$ and so the left hand side of \eqref{eq:BdyHSAction} makes sense and is naturally interpreted as a section of $\Lambda^* \Iie T^*H_X.$
It also establishes the equality \eqref{eq:BdyHSAction} since the Thom form over $H_M$ is the restriction of the Thom form over $M,$ and the push-forward is along the restriction of $\pi_X$ to $H_X.$

\end{proof}

%%%%%%%%%%%%%%%%%%%%%%%%%%%%%
\subsection{Pull-back of mezzoperversities}
%%%%%%%%%%%%%%%%%%%%%%%%%%%%%

Let $(\hat X, g_X)$ and $(\hat M, g_M)$ be smoothly stratified spaces endowed with $\iie$ metrics and resolutions $\wt X$ and 
$\wt M$ respectively. Recall from Definition \ref{def:StratPreserving} that a smooth codimension preserving map between them
\begin{equation*}
	F \in \CI_{b,cod}(\hat X, \hat M)
\end{equation*}
is a map $F \in \cC^0(\hat X, \hat M)$ that lifts to a smooth $b$-map $\wt F \in \CI(\wt X, \wt M)$ such that the codimension of a stratum of $\hat M$ coincides with the codimension of its inverse image in $\hat X.$ Given such a map, we use the boundary Hilsum-Skandalis maps to define the pull-back of a mezzoperversity from $\hat M$ to $\hat X$ and then show that it is a mezzoperversity on $\hat X.$

Let $Y^1_X, \ldots, Y^T_X$ be an ordering of the strata of $\hat X$ by non-decreasing depth, and let
\begin{equation*}
	Y^i_M \text{ be the stratum of $M$ containing $F(Y^i_X)$}
\end{equation*}
Note that the list $Y^1_M, \ldots, Y^T_M$ may repeat strata of $\hat M,$ and may not have all of the strata of $\hat M.$
Let $H^i_X$ and $H^i_M$ be the boundary hypersurfaces of $\wt X$ and $\wt M$ sitting above $Y^i_X$ and $Y^i_M$ respectively, and denote the fibrations by
\begin{equation*}
	Z^i_X - H^i_X \xlra{\phi^i_X} Y^i_X, \quad
	Z^i_M - H^i_M \xlra{\phi^i_M} Y^i_M.
\end{equation*}
For each stratum $Y^i_M,$ let $W^i_M = W(Y^i_M)$ be the flat subbundle of $\tH^{\mid}_{\cW_M(Z^i_M)}(H^i_M/Y^i_M)$ in $\cW_M.$
Also, let use denote by 
\begin{equation*}
	\cH^{\ell}_{\cW(Z^i_M)}(H^i_M/Y^i_M)
\end{equation*}
the bundle of vertical Hodge cohomology with boundary conditions.\\

Note that, at each stratum $Y^i_X$ and each $q \in Y^i_X,$ the induced map on the fibers
\begin{equation*}
	\wt F_q = \wt F\rest{(Z^i_X)_q} : (Z^i_X)_q \lra (Z^i_M)_{F(q)}
\end{equation*}
is itself a smooth codimension-preserving $b$-map. Indeed, it is clearly a smooth map and the `codimension-preserving' assumption is that it preserves the dimension of the links, but `the link of a link is a link' so this is automatic.

In particular, since the fibers $Z^1_X$ are depth zero stratified spaces so are their images, i.e., they are both smooth closed manifolds.
For $q\in Y^1_X,$ it follows from \cite{ALMP11} and is easy to see directly, that $HS(\wt F_q)$ induces a map in de Rham cohomology, $[HS(\wt F_q)],$ that only depends on the homotopy class of $\wt F_q.$

\begin{definition} We set:
\begin{equation*}
	F^{\sharp}( \tH^{\ell}(H^1_M/Y^1_M) )_q = [HS(\wt F_q)] \lrpar{ \tH^{\ell}( (Z^1_M)_{F(q)} ) } \subseteq \tH^{\ell}( (Z^1_X)_{q} ), 
\end{equation*}
\end{definition}
It follows from the previous definition that
\begin{equation*}
	F^{\sharp}( \tH^{\ell}(H^1_M/Y^1_M) ) = 
	\bigcup_{q \in Y^1_X} F^{\sharp}( \tH^{\ell}(H^1_M/Y^1_M) )_q
\end{equation*}
is a vector subbundle of $\tH^{\ell}( H^1_X/Y^1_X )$ over $Y^1_X.$
Notice that the pull-back maps $HS(\wt F_q)$ fit together to form $HS_{Y^1_X}(\wt F)$ which, from  \eqref{eq:expcoversid} and \eqref{eq:wtFcoversF} satisfies
\begin{equation*}
	HS_{Y^1_X}(\wt F) (\phi_M^1)^*\CI(Y^1_M, \Lambda^jT^*Y^1_M) \subseteq 
	(\phi_X^1)^*\CI(Y^1_X, \Lambda^jT^*Y^1_X).
\end{equation*}
%
%Indeed, from  \eqref{eq:expcoversid} and \eqref{eq:wtFcoversF} we see that the maps involved in $HS_{Y^1_X}(\wt F)$ preserve the fibers of the fibrations.
It follows, by Remark \ref{rmk:GaussManin}, that the map $HS_{Y^1_X}(\wt F)$ intertwines the Gauss-Manin connections.
Thus $F^{\sharp}( \tH^{\ell}(H^1_M/Y^1_M) )$ is a flat subbundle of  $\tH^{\ell}( H^1_X/Y^1_X )$ with its natural flat structure.

Reasoning in the same way, if we define
\begin{equation*}
	W(Y^1_X) = F^{\sharp}( W(Y^1_M) ) = 
	\bigcup_{q \in Y^1_X} [HS(\wt F_q)] \lrpar{ W(Y^1_M)_{F(q)}  }
\end{equation*}
then $W(Y^1_X)$ is a flat subbundle of  $\tH^{\mid}( H^1_X/Y^1_X )$ with its natural flat structure.
Moreover it follows from Lemma \ref{lem:RegPullBack} that 
\begin{equation*}
	HS(F) \cD_{\cW_M}^{\reg}(\eth_{\dR}^M) \subseteq \cD_{\max, \cW(Z^1_X)}(d^X)
\end{equation*}
where the space on the right is defined in \eqref{eq:DefDmaxW1}, and the space used on the left is defined in \cite[Lemma 2.2]{ALMP13.1} as
\begin{multline*}
	\cD_{\cW_M}^{\reg}(\eth_{\dR}^M) = \{ u \in \cD_{\cW_M}(\eth_{\dR}^M) : \text{ in a distinguished neighborhood $\cU_q = \bbB\times [0,1] \times Z$} \\
	\text{ around a point of a stratum $Y,$ the terms in its asymptotic expansion are smooth on $\bbB$} \}.
\end{multline*}

Now since $HS(F)$ intertwines $d^M$ and $d^X$ and is a bounded map on $L^2,$ it follows that we can replace the domain on the left with its graph closure with respect to $d^M$ and obtain
\begin{equation*}
	HS(F) \cD_{\cW_M}(d^M) \subseteq \cD_{\max, \cW(Z^1_X)}(d^X).
\end{equation*}
These considerations suggests the following definition and theorem, and indeed constitute most of the proof for depth one pseudomanifolds.

\begin{definition}
At the non-Witt strata of $X$ we define
\begin{equation*}
	W(Y^i_X) = F^{\sharp}W(Y^i_M) = 
	\bigcup_{q \in Y^i_X} HS(\wt F_q)( W(Y^i_M)_{F(q)} )
\end{equation*}
and we refer to the set
\begin{equation*}
	\cW_X = F^{\sharp}\cW_M = \{ W(Y^i_X) \}.
\end{equation*}
as the {\bf pull-back of the mezzoperversity $\cW_M$ by $F.$}
\end{definition}

\begin{theorem}\label{thm:UntwistedHS}
Let $\hat X,$ $\hat M,$ $F$ be as above.
The pull-back of a mezzoperversity $\cW_M$ on $M$ is itself a mezzoperversity on $X,$
\begin{equation*}
	\cW_X = F^{\sharp}\cW_M.
\end{equation*}
The Hilsum-Skandalis replacement of $F$ satisfies
\begin{equation*}
	HS(F) \cD_{\cW_M}(d^M) \subseteq \cD_{\cW_X}(d^X), \quad
	HS(F) d^M = d^X HS(F) 
\end{equation*}
and hence defines a map in de Rham cohomology
\begin{equation}\label{eq:HSCoho}
	[HS(F)]: \tH_{\cW_M}^j(\hat M) \lra \tH_{\cW_X}^j(\hat X), \quad \Mforevery j \in \bbN_0.
\end{equation}
This map is unchanged by varying $F$ smoothly through smooth codimension-preserving $b$-maps.
\end{theorem}

\begin{proof}
If $\hat X$ has depth zero, that is, if $\hat X$ is a smooth closed manifold, then since $F$ is codimension-preserving it must map into the smooth part of $\hat M$ and the theorem is classical. Let us inductively suppose that we have established the theorem for spaces of depth less than the depth of $\hat X.$

At each singular stratum $Y^i_X$ of $\hat X,$ point $q\in Y^i_X,$ and for each $\ell \in \bbN_0,$ define 
\begin{equation*}
	F^{\sharp}( \tH^{\ell}_{\cW_M(Z^i_M)}(H^i_M/Y^i_M) )_q = [HS(\wt F_q)] \lrpar{ \tH^{\ell}_{\cW_M(Z^i_M)}( (Z^i_M)_{F(q)} ) }
\end{equation*}
since $Z^i_X$ has depth less than that of $\hat X,$ by our inductive hypothesis this is a subspace of 
\begin{equation*}
	\tH^{\ell}_{\cW_X(Z^i_X)}( (Z^i_X)_{q} ).
\end{equation*}
By homotopy invariance these spaces together form a vector bundle
\begin{equation*}
	F^{\sharp}( \tH^{\ell}_{\cW_M}(H^i_M/Y^i_M) ) = 
	\bigcup_{q \in Y^i_X} F^{\sharp}( \tH^{\ell}_{\cW_M(Z^i_M)}(H^i_M/Y^i_M) )_q,
\end{equation*}
which is flat with respect to the natural flat structure of $\tH^{\ell}_{\cW_X(Z^i_X)}( H^i_X/Y^i_X ).$
It follows that at the non-Witt strata of $X,$ $W(Y^i_X)$ is a flat subbundle of $\tH^{\mid}_{\cW_X(Z^i_X)}( H^i_X/Y^i_X ),$ and so $\cW_X$ is a mezzoperversity on $\hat X.$

At each stratum of $\hat X,$ Lemma \ref{lem:RegPullBack} shows that 
\begin{equation*}
	HS(F) \cD_{\cW_M}^{\reg}(\eth_{\dR}^M) \subseteq \cD_{\cW_X}(d^X).
\end{equation*}
As above, we can replace the domain on the left with its graph closure with respect to $d^M$ and obtain
\begin{equation*}
	HS(F) \cD_{\cW_M}(d^M) \subseteq \cD_{\cW_X}(d^X).
\end{equation*}
Thus we have a map in de Rham cohomology \eqref{eq:HSCoho}. If we vary $F$ smoothly through smooth codimension-preserving $b$-maps, then we can use \cite[Lemma 9.1]{ALMP11} to see that the map in cohomology is unchanged. (Note that although \cite[Lemma 9.1]{ALMP11} is stated for homotopy equivalences, the construction does not use more than that the map is smooth and codimension preserving.)
\end{proof}

Finally, let us discuss the important case of a stratified homotopy equivalence.

\begin{theorem}
Let $(\hat X,g_X),$ $(\hat M,g_M),$ and $(\hat J, g_J)$ be pseudomanifolds endowed with $\iie$ metrics.
Let $F\in \CI_{b,cod}(\hat X,\hat M)$ and $G \in \CI_{b,cod}(\hat M,\hat J),$ and let $\cW_J$ be a de Rham mezzoperversity on $\hat J.$
Then we have
\begin{equation*}
	(G \circ F)^{\sharp}(\cW_J) = F^{\sharp}\lrpar{ G^{\sharp}(\cW_J) }
\end{equation*}
and a commutative diagram
\begin{equation*}
	\xymatrix{
	\tH_{\cW_J}^*(\hat J) \ar[rd]_-{[HS(G)]} \ar[rr]^-{[HS(G \circ F)]} & & \tH_{ (G \circ F)^{\sharp}\cW_J }(\hat X) \\
	& \tH_{G^{\sharp}\cW_J}^*(\hat M) \ar[ur]_-{[HS(F)]} & }.
\end{equation*}

In particular, if $F: \hat X \lra \hat M$ is a stratified homotopy equivalence then $[HS(F)]$ is an isomorphism.
\end{theorem}

\begin{proof}
This follows from \cite[Lemma 9.2]{ALMP11} where we showed that $[HS(G \circ F)] = [HS(F)] \circ [HS(G)].$
\end{proof}

%%%%%%%%%%%%%%%%%%%%%%%%
\subsection{Twisted cohomology}
%%%%%%%%%%%%%%%%%%%%%%%%

In this section we will explain how the constructions above define a pull-back map on twisted cohomology.

Let $(\hat M, g_M, \cW_M)$ be a pseudomanifold with an $\iie$ metric and de Rham mezzoperversity.
As in \S\ref{sec:TwistedDeRham}, let $\Gamma$ be a countable, finitely generated, finitely presented group, let $r:\hat M \lra B\Gamma$ be a classifying map that is smooth on $M,$ and $E \lra M$ an associated flat vector bundle.
As described above this defines a closed domain for the flat connection (which we think of as a twisted exterior derivative)
\begin{equation*}
	d^E: \cD_{\cW_M}(d^E) \subseteq L^2(M; \Lambda^*\Iie T^*M \otimes E) \lra L^2(M; \Lambda^*\Iie T^*M \otimes E)
\end{equation*}
with corresponding twisted de Rham cohomology groups, $\tH^*_{\cW_M}(\hat M;E).$

Now assume as in \S\ref{sec:IfsHom} that $(\hat X,g_X)$ is another pseudomanifold with an $\iie$ metric and that
\begin{equation*}
	F \in \CI_{b,cod}(\hat X,\hat M)
\end{equation*}
is a smooth codimension-preserving $b$-map.
The map $F \circ r:\hat X \lra B\Gamma$ defines a flat bundle $F^*E \lra X$ which is just the pull-back bundle of $E$ with its pulled-back connection.
We also have the pulled-back mezzoperversity $\cW_X = F^{\sharp}\cW_M$ and we can assume without loss of generality that the metric $g_X$ is adapted to $\cW_X.$

It is straightforward to see that the Hilsum-Skandalis replacement of $F$ constructed above extends to a map
\begin{equation*}
	HS(F): \CI(M; \Lambda^*\Iie T^*M \otimes E) \lra \CI(X; \Lambda^*\Iie T^*X \otimes F^*E).
\end{equation*}
Directly from its definition, the domain $\cD_{\cW_M}(d^E)$ near a boundary hypersurface is just $N$ copies of the domain $\cD_{\cW_M}(d^M),$
so from Theorem \ref{thm:UntwistedHS} we can see that $HS(F)$ will map it into $\cD_{\cW_X}(d^{F^*E}),$ and thus ultimately induce a map in cohomology.

Just as above, the same considerations work if we use the canonical representation of $\Gamma$ on $C^*_r\Gamma;$ we end up with a pull-back map from the differential forms on $M$ twisted by $\sG(r)$ to the differential forms on $X$ twisted by $\sG(F\circ r).$

\begin{theorem}\label{thm:TwistedCohoIso}
Let $\hat X,$ $\hat M,$ $F,$ and $E$ be as above.
The Hilsum-Skandalis replacement of $F$ satisfies
\begin{equation*}
	HS(F) \cD_{\cW_M}(d^E) \subseteq \cD_{\cW_X}(d^{F^*E}), \quad
	HS(F) d^E = d^{F^*E} HS(F) 
\end{equation*}
and the corresponding map in de Rham cohomology
\begin{equation*}
	[HS(F)]: \tH_{\cW_M}^j(\hat M;E) \lra \tH_{\cW_X}^j(\hat X;F^*E), \quad \Mforevery j \in \bbN_0
\end{equation*}
is unchanged by varying $F$ smoothly through smooth codimension-preserving $b$-maps and is an isomorphism if $F$ is a stratified homotopy equivalence.

Another natural extension of $HS(F)$ satisfies
\begin{equation*}
	HS(F) \cD_{\cW_M}(d^{\sG(r)}) \subseteq \cD_{\cW_X}(d^{\sG(F \circ r)}), \quad
	HS(F) d^{\sG(r)} = d^{\sG(F\circ r)} HS(F) 
\end{equation*}
and the induced map
\begin{equation*}
	HS(F): \ker d^{\sG(r)} / \Im d^{\sG(r)} \lra \ker d^{\sG(F\circ r)} / \Im d^{\sG(F\circ r)}
\end{equation*}
is unchanged by varying $F$ smoothly through smooth codimension-preserving $b$-maps and is an isomorphism if $F$ is a stratified homotopy equivalence.
\end{theorem}

%%%%%%%%%%%%%%%%%%%%%%%%
\section{The analytic signature of a Cheeger space} \label{sec:Sign}
%%%%%%%%%%%%%%%%%%%%%%%%

Let $(\hat X,g_X)$ be an oriented pseudomanifold with an $\iie$ metric, and define the intersection pairing of $L^2$-differential forms by
\begin{equation*}
\xymatrix @R=1pt
{ L^2(X; \Lambda^*\Iie T^*X) \times L^2(X; \Lambda^*\Iie T^*X) \ar[r]^-{Q} & \bbR \\
(u,v) \ar@{|->}[r] & \int_X u \wedge v }.
\end{equation*}
Every mezzoperversity $\cW$ on $X$ has an associated dual mezzoperversity $\sD\cW$ \cite[\S 6]{ALMP13.1}.
The simplest characterization of $\sD\cW$ is to point out that 
if $g_X$ is adapted to $\cW$ then the Hodge star induces a bijection
\begin{equation*}
	*: \cD_{\cW}(d) \lra \cD_{\sD\cW}(d).
\end{equation*}

The restriction of $Q,$
\begin{equation*}
	Q: \cD_{\cW}(d) \times \cD_{\sD\cW}(d) \lra \bbR,
\end{equation*}
is non-degenerate and descends to a non-degenerate pairing
\begin{equation*}
	Q: \tH_{\cW}^*(\hat X) \times \tH_{\sD\cW}^*(\hat X) \lra \bbR.
\end{equation*}
We can think of this as a refinement of Goresky-MacPherson's generalized Poincar\'e duality.

\begin{definition}\label{def:mezzo}
A mezzoperversity $\cW$ is {\bf self-dual} if $\cW = \sD\cW.$
A pseudomanifold $\hat X$ on which there is a self-dual mezzoperversity is a {\bf Cheeger space.}
\end{definition}

Not all pseudomanifolds $\hat X$ are Cheeger spaces. For example a space with isolated conic singularities is a Cheeger space if and only if the signatures of the links vanish. The reductive Borel-Serre compactification of a Hilbert modular surface is a more intricate example of a Cheeger space, see \cite{Banagl-Kulkarni, ABLMP}. \\

Let $\hat X$ be a Cheeger space and $\cW$  a self-dual mezzoperversity with an adapted $\iie$ metric $g.$
In this case the intersection pairing is a non-degenerate pairing 
\begin{equation}\label{eq:NonDegQ}
	Q: \tH_{\cW}^*(\hat X) \times \tH_{\cW}^*(\hat X) \lra \bbR\,.
\end{equation}

Moreover, as already pointed out, the  Hodge star is now a linear map
\begin{equation}\label{eq:self}
	*: \cD_{\cW}(\eth_{\dR}) \lra \cD_{\cW}(\eth_{\dR}).
\end{equation}

This allows for the definition of the signature operator, as we shall now explain.
If $\hat X$ is even-dimensional, the Hodge star induces a natural involution on the compactly supported differential forms on $X,$
\begin{equation}\label{defcI}
	\cI: \CIc(X; \Lambda^*T^*X) \lra \CIc(X;\Lambda^*T^*X), \quad \cI^2 =\Id
\end{equation}
that extends to $L^2$ $\iie$ forms
\begin{equation*}
	\cI: L^2(X; \Lambda^*(\Iie T^*X)) \lra \CIc(X;\Lambda^*(\Iie T^*X)), \quad \cI^2 =\Id.
\end{equation*}

The $+1,$ $-1$ eigenspaces define the self-dual and anti-self-dual forms,
\begin{equation*}
	L^2(X; \Lambda_+^*(\Iie T^*X)), \quad
	\Mand
	L^2(X; \Lambda_-^*(\Iie T^*X)).
\end{equation*}
The de Rham operator, extended to complexified forms, anticommutes with $\cI$ and this defines, as usual,
the signature operator 
\begin{equation*}
	\eth_{\sign}   
	= \begin{pmatrix} 0 & \eth_{\sign}^- \\ \eth_{\sign}^+ & 0 \end{pmatrix}.
\end{equation*}
If $\hat X$ is odd-dimensional, the signature operator is
\begin{equation*}
	\eth_{\sign} = -i(d\cI + \cI d)= -i \cI (d-\delta)= -i (d-\delta) \cI
\end{equation*}

To summarize, the signature operator is an unbounded operator with domain equal to complex-valued compactly supported
differential forms, suitably graded in the even dimensional case.

If $\cW$ is a {\it self-dual} mezzoperversity and $X$ is even-dimensional, then by \eqref{eq:self},  we can define
\begin{equation*}
	\cD_{\mathcal{W}}(\eth_{\sign}^{\pm}) = \cD_{\mathcal{W}}(\eth_{\dR}) \cap L^2(X; \Lambda_\pm^*(\Iie T^*X)).
\end{equation*}

We then obtain \cite[Theorem 7.6]{ALMP13.1}:
\begin{theorem}
Let $(\hat X, g)$ be an even dimensional  stratified space endowed with a suitably scaled $\iie$ metric $g$ and Cheeger ideal boundary conditions corresponding to a self-dual mezzoperversity $\cW.$
The signature operator 
\begin{equation*}
	\eth_{\sign}^+: \cD_{\mathcal{W}}(\eth_{\sign}^+) \subseteq L^2(X; \Lambda_+^*(\Iie T^*X)) \lra L^2(X; \Lambda_-^*(\Iie T^*X))
\end{equation*}
is closed and Fredholm, with adjoint  $(\eth_{\sign}^-, \cD_{\mathcal{W}}(\eth_{\sign}^-))$. Its Fredholm index is equal to the signature of the generalized 
Poincar\'e duality quadratic form $Q$ from \eqref{eq:NonDegQ}.
\end{theorem}

If $\hat X$ is odd dimensional we define 
\begin{equation*}
	\cD_{\mathcal{W}}(\eth_{\sign}) = \cD_{\mathcal{W}}(\eth_{\dR})
\end{equation*}
Using \eqref{eq:intersec} and \eqref{eq:self} we see immediately that $(\eth_{\sign}, \cD_{\mathcal{W}}(\eth_{\sign}))$ is well defined,
self-adjoint and Fredholm.

In this section we will show that this signature depends only on the Cheeger space $\hat X$ itself.
We will also prove analogous statements for the twisted `higher' signature operator.

{\em From now on we assume that all spaces are oriented and all representations are over the field of complex numbers.}

%%%%%%%%%%%%%%%%%%%%%%%%
\subsection{K-homology classes} $ $\\
%%%%%%%%%%%%%%%%%%%%%%%%

We start by using the results from \cite{ALMP11} recalled above to see that given $(\hat X, g, \cW),$ a Cheeger space with a self-dual mezzoperversity and an adapted $\iie$-metric, the corresponding signature operator defines a $K$-homology class in $\tK_*(\hat X) = \tK\tK_*(\cC(\hat X), \bbC).$ This extends the construction from \cite{ALMP11} and   \cite{Moscovici-Wu} for Witt spaces.

Recall \cite{Baaj-Julg, Blackadar} that an even unbounded Fredholm module for $\cC(\hat X)$ is a pair $(H,D)$ satisfying
\begin{enumerate}
\item $D$ is a self-adjoint unbounded operator acting on the Hilbert space $H,$ which also has an action (unitary $*$-representation) of $\cC(\hat X);$
\item $\cC(\hat X)$ has a dense $*$-subalgebra $\cA$ whose action preserves the domain of $D$ and satisfies
\begin{equation*}
	a \in \cA \implies [D,a] \text{ extends to a bounded operator on } H;
\end{equation*}
\item $\Id+D^2$ is invertible with compact inverse;
\item There is a self-adjoint involution $\tau$ on $H$ commuting with $\cC(\hat X)$ and anti-commuting with $D.$
\end{enumerate}
An odd Fredholm module for $\cC(\hat X)$ is a pair $(H,D)$ satisfying all but the last of these conditions.

In view of this definition we point out that the class of functions $\CI(\hat X)$ from Definition \ref{def:SmoothFun} form a dense subset of $\cC(\hat X)$ (e.g., by the Stone-Weierstrass theorem). We also point out that, from \cite[Theorem 6.6]{ALMP13.1}, if $D = (\eth_{\sign}, \cD_{\cW}(\eth_{\sign})),$ then 
\begin{equation}\label{eq:trace}
	 (\Id + D^2)^{-n-1} \text{ is trace-class on } L^2(X; \Lambda^*_{\bbC} \Iie T^*X).
\end{equation}

\begin{theorem}\label{theo:k-homology}
Let $(\hat X, g, \cW)$ be a Cheeger space with a self-dual mezzoperversity $\cW$ and adapted $\iie$ metric $g.$
The signature operator $(\eth_{\sign}, \cD_{\cW}(\eth_{\sign}))$
defines an unbounded Fredholm module
for $\cC(\widehat{X})$ and thus a class
$[\eth_{\sign,\cW}]\in KK_* (C(\widehat{X}),\bbC)$, $* \equiv\dim X \,{\rm mod} \,2$.
The class  $[\eth_{\sign,\cW}]$ does not change under a continuous homotopy of metrics and self-dual mezzoperversities.
\end{theorem}

\begin{proof}
In the notation above, we take $H=L^2(X;\Iie\Lambda^*X)$, endowed with the natural representation of $C(\widehat{X})$
by multiplication operators, $\cA = \CI(\hat X),$  $D = (\eth_{\sign}, \cD_{\cW}(\eth_{\sign})),$ and $\tau = \cI$ from \eqref{defcI}.
All the conditions defining an unbounded Kasparov module are easily proved using the results
of the previous section. \\
First, note that multiplication by any element  $f$ of $\cA$ preserves $\cD_{\cW}(\eth_{\sign}).$
Next, note that $[\eth_{\sign},f]$ is given  by Clifford multiplication by $df$ which exists  everywhere and is an element in $L^\infty (\widehat{X})$; in particular $[\eth_{\sign},f]$ 
extends to a bounded operator on $H$; finally we know that $(1+D^2)^{-1}$ is a compact operator (indeed, the compact inclusion of $\cD_{\cW}$ into $L^2,$ together with self-adjointness proves that both $(i+D)^{-1}$ and $(-i+D)^{-1}$ are compact). Moreover we recall equation \eqref{eq:trace}. 
Thus there is a well defined class in $KK_* (C(\hat{X}),\bbC)$ which we denote simply by $[\eth_{\sign,\cW}].$

Given a homotopy $(g_t, \cW_t)$ of self-dual mezzoperversities and adapted $\iie$ metrics,
let $\eth^t_{\sign,\cW_t}$ be the corresponding signature operators, with domains in $H=H_t.$
Proceeding as in the work of Hilsum on Lipschitz manifolds \cite{Hilsum-LNM} one can prove that the 1-parameter family
$(H_t, \eth^t_{\sign,\cW_t})$ defines an {\it unbounded} operatorial  homotopy; using the homotopy invariance of $\tK\tK$-theory one obtains 
$$[\eth^0_{\sign, \cW_1}]=[\eth^1_{\sign,\cW_2}] \;\,\text{ in }\;\;  \tK\tK_* (C(\widehat{X}),\bbC)\,.$$
We omit the details since they are a repetition of the ones given in  \cite{Hilsum-LNM}. 
\end{proof}

We can carry out the same construction for the signature operator twisted by a flat $\bbC$-vector bundle as we have considered above.
However the machinery of $\tK$-homology makes this unnecessary; indeed the class $[\eth_{\sign, \cW}]$ determines the index of the signature operator with coefficients in $E.$

Coefficients in a bundle of $C^*$-algebra modules does require a different, though formally very similar, construction.
Thus let $(\hat X, g, \cW)$ be a Cheeger space with a self-dual mezzoperversity and an adapted $\iie$ metric, and let $r:\hat X \lra B\Gamma$ be a classifying map for a $\Gamma$-covering with $\Gamma$ a finitely generated discrete group. Let $\sG(r)$ be the flat bundle associated to $r$ with fiber $C^*_r\Gamma.$

First note that exactly as above we have a twisted signature operator, $(\eth_{\sign}^{\sG(r)}, \cD_{\cW}(\eth_{\sign}^{\sG(r)})),$ which by Theorem \ref{thm:HigherSA} is closed, Fredholm, and self-adjoint and has domain is $C^*_r\Gamma$-compactly included in $L^2(X; \Lambda^*\Iie T^*X \otimes \sG(r)).$ As in \cite[Proposition 6.4]{ALMP11} these facts can be combined with arguments of Skandalis, published in \cite{Rosenberg-Weinberger}, to prove the following theorem:

\begin{theorem}
Let $(\hat X, g, \cW),$ $r:\hat X \lra B\Gamma,$ and $\sG(r)$ be as above.
\begin{itemize} 
\item If $D$ is the operator $(\eth_{\sign}^{\sG(r)}, \cD_{\cW}(\eth_{\sign}^{\sG(r)})),$ then $\Id + D^2$ is surjective (i.e., $D$ is a `regular' operator).  
\item The operator $D$ and the $C^*_r\Gamma$-Hilbert module $L^2(X; \Lambda^*\Iie T^*X \otimes \sG(r))$ define an unbounded Kasparov
$(\bbC, C^*_r\Gamma)$-bimodule and thus a class in $\tK_*(C^*_r\Gamma).$ We call this the index class associated to $D$ and denote it by
\begin{equation*}
	\Ind(\eth_{\sign,\cW}^{\sG(r)}) \in KK_*( \bbC, C^*_r\Gamma ) \simeq K_*(C^*_r\Gamma).
\end{equation*}
\item
If 
\begin{equation*}
	[[\eth_{\sign, \cW}]] \in \tK\tK_*(C(\hat X)\otimes C^*_r\Gamma, C^*_r\Gamma)
\end{equation*}
is the class obtained by tensoring the class $[\eth_{\sign, \cW}] \in \tK\tK_*(\cC(\hat X), \bbC)$ with $\Id_{C^*_r\Gamma},$ then
\begin{equation*}
	\Ind(\eth_{\sign,\cW}^{\sG(r)}) = [\sG(r)] \otimes [[\eth_{\sign,\cW}]].
\end{equation*}
That is, the index class of the higher signature operator is the Kasparov product of $[[\eth_{\sign, \cW}]]$ with $[\sG(r)] \in \tK\tK_0(\bbC, \cC(\hat X) \otimes C^*_r\Gamma).$
\end{itemize}
\end{theorem}

%\begin{proof}
%As in \cite{ALMP11}, this is shown following arguments of Skandalis published in \cite{Rosenberg-Weinberger}.
%We very briefly sketch the proof and refer the reader to \cite[Proposition 6.4]{ALMP11} for more details.
%
%The thrust of the proof is to compare $(\eth_{\sign}^{\sG(r)}, \cD_{\cW}(\eth_{\sign}^{\sG(r)})),$ with the twist of $(\eth_{\sign}, \cD_{\cW}(\eth_{\sign}))$
%We start with the class $[\eth_{\sign, \cW}]$ in $\tK\tK_*(\cC(\hat X), \bbC)$ and tensor it with $\Id_{C^*_r\Gamma}$ to obtain an unbounded Kasparov 
%$(\cC(\hat X) \otimes C^*_r\Gamma, C^*_r\Gamma)$-bimodule which we denote $(\cE, D),$ and a corresponding $\tK\tK$-class
%%
%\begin{equation*}
%	[[\eth_{\sign, \cW}]] \in \tK\tK_*(C(\hat X)\otimes C^*_r\Gamma, C^*_r\Gamma).
%\end{equation*}
%%
%Now the key fact is that $\eth_{\sign}^{\sG(r)}$ defines a $D$-connection in the sense of Skandalis.
%\end{proof}

\begin{corollary}\label{cor:assembly}
Let $\beta:K_* (B\Gamma)\to K_* (C^*_r\Gamma)$ be the assembly map; let $r_* [\eth_{\sign,\cW}]\in K_* (B\Gamma)$
the push-forward of the signature K-homology class. Then 
\begin{equation}\label{assembly}
\beta(r_* [\eth_{\sign,\cW}])=\Ind (\eth_{\sign,\cW}^{\sG(r)}) \text{ in }  K_* (C^*_r\Gamma).
\end{equation}
\end{corollary}

\begin{proof}
Since $\Ind (\eth_{\sign,\cW}^{\sG(r)})= [\widetilde{C^*_r}\Gamma]\otimes [[\eth_{\sign,\cW}]]$, this follows immediately
from the very definition of the assembly map, see  \cite{Kasparov-inventiones,Kasparov-contemporary}.
\end{proof}

We point out that so far all of our constructions depend on $(g, \cW),$ though they are unchanged under  a homotopy of this data.

%%%%%%%%%%%%%%%%%%%%%%%%
\subsection{Stratified homotopy invariance of the analytic signature} \label{sec:StratHom} $ $\\
%%%%%%%%%%%%%%%%%%%%%%%%

Let $(\hat M, g_M, \cW_M)$ be an oriented Cheeger space with a self-dual mezzoperversity and an adapted $\iie$-metric.
Let $\hat X$ be an oriented Cheeger space and $F: \hat X \lra \hat M$ an orientation-preserving stratified homotopy equivalence,
$\cW_X = F^{\sharp}(\cW_M)$ and $g_X$ an adapted $\iie$-metric.
We have already shown that the de Rham cohomology groups are stratified homotopy invariant, so that
\begin{equation*}
	[HS(F)]: \tH^*_{\cW_M}(\hat M) \xlra{\cong} \tH^*_{\cW_X}(\hat X).
\end{equation*}
Each of these cohomology spaces has a quadratic form \eqref{eq:NonDegQ} and it is {\em a priori} not clear what if any relation the two quadratic forms will have.
This is, however, exactly the situation for which the Hilsum-Skandalis replacement was formulated. We can use \cite[Proposition 9.3]{ALMP11} (cf. \cite[Lemma 2.1]{Hilsum-Skandalis}) to see that the signature of $\tH^*_{\cW_X}(\hat X)$ is equal to the signature of $\tH^*_{\cW_M}(\hat M)$ with their respective quadratic forms.

We also have the corresponding result for the signature operator twisted by a flat bundle of projective $C^*$-modules. 
\begin{theorem}\label{thm:StratHomInvSign}
Let $(\hat M, g_M, \cW_M)$ be an oriented Cheeger space with a self-dual mezzoperversity and an adapted $\iie$-metric.
Let $\hat X$ be an oriented Cheeger space and $F: \hat X \lra \hat M$ an orientation-preserving stratified homotopy equivalence,
$\cW_X = F^{\sharp}(\cW_X)$ and $g_X$ an adapted $\iie$-metric.
Let $\Gamma$ be a countable, finitely generated, finitely presented group and let $r:\hat M \lra B\Gamma$ be a classifying map.
We have an equality
\begin{equation*}
	\Ind(\eth_{\sign,\cW_M}^{\sG(r)}) = \Ind(\eth_{\sign,\cW_X}^{\sG(F\circ r)}) \Min K_*(C^*_r\Gamma).
\end{equation*}
\end{theorem}

\begin{proof}
As in \S 9.4 of \cite{ALMP11}, let
\begin{equation*}
\begin{gathered}
	Q_X: L^2(X; \Lambda^* \Iie T^*X \otimes \sG(r)) \times L^2(X; \Lambda^* \Iie T^*X \otimes \sG(r)) \lra C^*_r\Gamma \\
	Q_X(u,v) = \int_X u \wedge v^*
\end{gathered}
\end{equation*}
and denote the adjoint of an operator $T$ with respect to $Q_X$ by $T'.$

From  \cite[Lemma 2.1]{Hilsum-Skandalis}, it suffices to show that:\\
a) $HS(F)d^{\sG(r)} = d^{\sG(F \circ r)} HS(F)$ and $HS(F)(\cD_{\cW}(d^{\sG(r)})) \subseteq \cD_{F^{\sharp}\cW}(d^{\sG(F \circ r)})$\\
b) $HS(F)$ induces an isomorphism 
$HS(F): \ker d^{\sG(r)} / \Im d^{\sG(r)} \lra \ker d^{\sG(F\circ r)} / \Im d^{\sG(F\circ r)}$\\
c) There is a bounded operator $\Upsilon$ on a Hilbert module associated to $\hat M,$ acting on $\cD_{\cW}(d^{\sG(r)}),$ such that
$\Id - HS(F)'HS(F) =  d^{\sG(r)} \Upsilon + \Upsilon d^{\sG(r)}$\\
d) There is a bounded involution $\eps$ on $\hat M,$ acting on $\cD_{\cW}(d^{\sG(r)})$ commuting with $\Id - HS(F)'HS(F)$ and anti-commuting with $d^{\sG(r)}.$ 

We have established ($a$) and ($b$) in Theorem \ref{thm:TwistedCohoIso}. 
The computations in the proof of Proposition 9.3 in \cite{ALMP11} show that  $\Id - HS(F)'HS(F) =  d^{\sG(r)} \Upsilon + \Upsilon d^{\sG(r)},$ with $\Upsilon$ constructed in \cite[Lemma 9.1]{ALMP11}, and we can check that it preserves $\cD_{\cW}(d^{\sG(r)}).$ Finally, for ($d$), it suffices to take $\eps u = (-1)^{|u|}u,$ on all forms $u$ of pure differential forms degree $|u|.$
\end{proof}

%%%%%%%%%%%%%%%%%%%%%%%%
\subsection{Bordism invariance of the analytic signature} \label{sec:Bordism} $ $\\
%%%%%%%%%%%%%%%%%%%%%%%%
We will define a bordism between stratified pseudomanifolds with self-dual mezzoperversities, that is Cheeger spaces.
The corresponding topological object was introduced by Banagl in \cite{BanaglShort} where it was denoted $\Omega_*^{SD},$ and was later considered by Minatta \cite{Minatta} where it was denoted $\mathrm{Sig},$ short for `signature homology.' 

\begin{definition}
If $M$ is a topological space, we denote by $\Sig{n}{M}$ the bordism group of four-tuples $(\hat X, g, \cW, r:\hat X \lra M)$ where $\hat X$ is an oriented Cheeger space of dimension $n,$ $\cW$ is a self-dual Hodge mezzoperversity, $g$ is an adapted $\iie$ metric, and $r:\hat X \lra M$ is a continuous map.

An admissible bordism between $(\hat X, g, \cW, r:\hat X\lra M)$ and $(\hat X',g',\cW', r':\hat X' \lra M)$ is a four-tuple $(\sX, G, \sW, R:\sX \lra M)$ consisting of:\\
i) a smoothly stratified, oriented, compact pseudomanifold with boundary $\sX,$ whose boundary is $\hat X \sqcup \hat X',$
and whose strata near the boundary are collars of the strata of $\hat X$ or $\hat X',$\\
ii) an $\iie$ metric $G$ on $\sX$ that near the boundary
is of the collared form $dx^2 +g$ or $dx^2 + g',$ \\
iii) an adapted self-dual mezzoperversity $\sW$ that extends, in a collared way, that of $\hat X$ and $\hat X',$\\
iv) a map $R: \sX \lra M$ that extends $r$ and $r'.$
\end{definition}

If $M$ is a point, we will leave off the trivial maps $r:\hat X \lra \{\pt\}$ from the descriptions of the elements of $\Sig{n}{\pt}.$

\begin{theorem}\label{thm:CobInvSign}
If $(\hat X,g,\cW)$ and $(\hat X',g',\cW')$ are $n$-dimensional and cobordant then the indices of the corresponding signature operators coincide,
\begin{equation*}
	\ind (\eth_{\sign}^{X}, \cD_{\cW}(\eth_{\sign}^X)) = \ind (\eth_{\sign}^{X'}, \cD_{\cW'}(\eth_{\sign}^{X'}))
\end{equation*}
and so this Fredholm index defines a map
\begin{equation*}
	\Sig{n}{\pt} \lra \bbZ.
\end{equation*}

Similarly, for any finitely generated discrete group, $\Gamma,$ the index class of the twisted signature operators defines a map
\begin{equation}\label{eq:ansymsign}
	\Sig{n}{B\Gamma} \lra \tK_n(C^*_r\Gamma; \bbQ)
\end{equation}
where the latter group is $\tK_n(C^*_r\Gamma) \otimes_{\bbZ} \bbQ.$
\end{theorem}

We point out that the spaces $\Sig{n}{M}$ form an Abelian group (under disjoint union) and the analytic signature maps from 
this theorem are group homomorphisms.

\begin{proof}
We proceed as in \cite[\S 7]{ALMP11}. We shall only deal with the case where $n$ is even, the case $n$ odd being actually simpler.
Let $(\sX, G, \sW)$ be a bordism between $(\hat X,g,\cW)$ and $(\hat X',g',\cW'),$ where dim $ X= n$ ($X$ denotes the regular part of $\hat X$). Let $\cC_0(\sX)$ 
denote the vector space of continuous functions on $\sX$ which are zero on the boundary $\partial \sX$.
To the semisplit short exact sequence
$$
0 \rightarrow \cC_0(\sX) \rightarrow \cC(\sX)  \rightarrow \cC(\partial \sX) \rightarrow 0\,
$$ and the inclusion $i: \partial \sX \rightarrow \sX$, one associates the following long exact sequence
$$
\ldots \rightarrow KK^1(\cC_0(\sX), C^*_r \Gamma  ) \overset{\delta}\rightarrow KK^0 ( \cC(\partial \sX), C^*_r \Gamma ) 
 \overset{i_*}\rightarrow KK^0 ( \cC(\sX), C^*_r \Gamma ) \rightarrow \ldots
$$ Thus $ i_* \circ \delta =0$. First we apply this exact sequence to the case where $\Gamma =\{1\}$, 
$C^*_r \{1\} = \bbC$. 

Since by definition the elements of $\cC_0(\sX)$ vanish at $\partial \sX$, 
 we can proceed as in the proof of Theorem \ref{theo:k-homology} in order to show  that $\eth_{\sign,\sW}^{\sX}$ defines an unbounded 
Fredholm $(\cC_0(\sX), \bbC)$ bi-module and thus a class $[\eth_{\sign,\sW}^{\sX}] $ in 
$\tK\tK_1(\cC_0(\sX), \bbC)= \tK_1(\sX, \partial \sX).$
Since the metric and the self-dual mezzoperversities are ``collared" near the boundary of $\sX,$
the same proof as in \cite[\S 7]{ALMP11} applies to show that
\begin{equation*}
	\delta[\eth_{\sign,\sW}^{\sX}] = 2 [\eth_{\sign, \cW}^{X}] \oplus 2 [ -\eth_{\sign, \cW'}^{X'} ]
\end{equation*}
%%
%where $\delta$ is the map
%%
%\begin{equation*}
%	\delta: \tK_1(\sX, \partial \sX) \lra \tK_0(\partial \sX)
%\end{equation*}
%%
%in the long exact sequence induced by $i: \partial \sX \hookrightarrow \sX.$
Consider the constant map $\pi: \sX \rightarrow \{ pt \}$, we write its restriction 
to $\partial \sX$ under the form $\pi^\partial= \pi \circ i$. We have a natural map 
$$\pi^\partial_*: KK_0( C( \hat X) , \bbC) \rightarrow KK_0( \bbC , \bbC)\simeq \Z\,,$$ 
and similarly for $KK_0( C( \hat X') , \bbC)$.
The index of $\eth_{\sign, \cW}^{X}$ (or the signature of $(X,g,\cW)$) is then equal to 
$\pi^\partial_*( [\eth_{\sign, \cW}^{X}])$. Recall that $ i_*\delta=0$.
Therefore, the difference of twice the signatures of $(X,g,\cW)$ and $(X',g',\cW')$ is given by
\begin{equation*}
	\pi^{\pa}_* \left( 2  [\eth_{\sign, \cW}^{X}] \oplus 2 [ -\eth_{\sign, \cW'}^{X'} ] \right)
	= \pi_*^{\pa}\delta[\eth_{\sign,\sW}^{\sX}] 
	= \pi_*i_*\delta[\eth_{\sign,\sW}^{\sX}] =0\,.
\end{equation*}
Since $\bbZ$ has no torsion, we are done.

A similar argument \cite[\S 7]{ALMP11} works in the presence of a finitely generated discrete group, $\Gamma.$
Given $(\sX, G, \sW, R:\sX \lra B\Gamma),$ let $\pa\sW = \sW\rest{\pa \sX}$ and $\pa R = R\rest{\pa \sX}.$
We argue as above but now using that
$(\sX, G, \sW, R:\sX \lra B\Gamma)$ defines a class $[\eth_{\sign,\sW}^{\sG(R)}] \in \tK\tK^1(\cC_0(\sX), C^*_r\Gamma),$
the exactness of 
\begin{equation*}
	\tK\tK^1( \cC_{0}(\sX), C^*_r\Gamma) \xlra{\delta}
	\tK\tK^0( \cC(\pa \sX), C^*_r\Gamma) \xlra{i_*}
	\tK\tK^0(\cC(\sX), C^*_r\Gamma),
\end{equation*}
and the fact that $\delta[\eth_{\sign,\sW}^{\sG(R)}]  = [2\eth_{\sign, \pa\sW}^{\sG(\pa R)}].$
\end{proof}

%%%%%%%%%%%%%%%%%%%%%%%%
\subsection{The analytic signature does not depend on the mezzoperversity} \label{SignIndep} $ $\\
%%%%%%%%%%%%%%%%%%%%%%%%
In \cite{Banagl:LClasses}, Banagl gave a clever argument to show that the signature of a self-dual sheaf on an $L$-space (a topological 
version of a Cheeger space, see below and \cite{ABLMP}) is the same for every self-dual sheaf. We now show that the same argument 
works in the analytic setting, and we will later make use of this result to connect the topological and analytic signatures. In particular, 
we show that if $(\hat X, g, \cW)$ is a Cheeger space with a self-dual mezzoperversity $\cW$ and an adapted metric, then the 
analytic signature depends only on $\hat X$ and not on $g$ or $\cW.$

\begin{theorem}
Let $\hat X$ be a Cheeger space and let $(g, \cW)$ and $(g', \cW')$ be two pairs of self-dual Hodge mezzoperversities with 
adapted $\iie$-metrics. For any topological space, $M,$ and a map $r:\hat X \lra M,$ 
\begin{equation*}
	(\hat X, g, \cW, r:\hat X \lra M) \text{ is cobordant to } (\hat X, g', \cW', r:\hat X \lra M).
\end{equation*}
\end{theorem}

\begin{proof}
Let us consider the pseudomanifold with boundary 
\begin{equation*}
	\sX = \hat X \times [0,1]_t.
\end{equation*}
Instead of the product stratification, let us stratify $\sX$ using the strata of $\hat X$ as follows:\\
i) The regular stratum of $\hat X$ contributes $X \times [0,1]$\\
ii) Every singular stratum of $\hat X,$ $Y^k,$ contributes three strata to $\sX,$
\begin{equation*}
	Y^k \times [0,1/2), \quad Y^k \times (1/2, 1], \quad Y^k \times \{1/2 \}.
\end{equation*}
Notice that the link of $\sX$ at $Y^k \times [0,1/2)$ and $Y^k \times (1/2, 1]$ is equal to $Z^k,$
while, a neighborhood of $Y^k \times \{1/2\}$ in $\sX$ fibers over $Y^k \times \{1/2\}$ with fiber
$\bbR \times C(Z^k),$ and so
\begin{equation*}
	\text{ the link of $\sX$ at $Y^k \times \{ 1/2 \}$ is the (unreduced) suspension of $Z^k,$ $S Z^k.$ }
\end{equation*}
The lower middle perversity intersection homology of $S Z^k,$ when $\dim Z^k = 2j-1,$ is given by
\cite[Pg. 6]{Friedman:Coeff}
\begin{equation*}
	I^{\bar m}H_i(S Z^k) = 
	\begin{cases}
		I^{\bar m}H_{i-1}(Z^k) & i>j \\
		0 & i=j \\
		I^{\bar m}H_i(Z^k) & i <k
	\end{cases}
\end{equation*}
so $\sX$ always satisfies the Witt condition at the strata $Y^k \times \{ 1/2 \}.$

Let us endow $\sX$ with any $\iie$ metric $G$ such that,
for some $t_0>0,$ 
\begin{equation*}
	G\rest{X \times [0,t_0)} = g + dt^2, \quad
	G\rest{X \times (1-t_0, 1]} = g' + dt^2.
\end{equation*}

Next we need to endow $\sX$ with a self-dual mezzoperversity $\sW.$
Let $Y^1, \ldots, Y^T$ be an ordering of the strata of $\hat X$ with non-decreasing depth.
Denote 
\begin{equation*}
	\cW = \{ W^1 \lra Y^1, \ldots, W^T \lra Y^T \}, \quad
	\cW' = \{ (W^1)' \lra Y^1, \ldots, (W^T)' \lra Y^T \}
\end{equation*}
and denote the fiber of, e.g., $W^j \lra Y^j$ at the point $q \in Y^j,$ by $W^j_q.$
Let us define
\begin{equation*}
	W^1_- \lra Y^1 \times [0,1/2)
\end{equation*}
by requiring that the Hodge-de Rham isomorphism identifies all of the fibers. Notice that the vertical $L^2$-de Rham cohomology is constant along the fibers of $H^1 \times [0,1/2) \lra Y^1 \times [0,1/2),$ by stratified homotopy invariance, so this condition makes sense and determines a flat, self-dual, vector bundle. (Where we are using that self-duality can be checked at the level of de Rham cohomology.) It may be necessary to scale the metric to make it compatible with $W^1_-,$ but this can always be done without changing the metric in a collar neighborhood of $\pa \sX,$ since the original metrics and mezzoperversities are initially adapted.
Once this is done, we can define $W^2_- \lra Y^2 \times [0,1/2)$ in the same way, and inductively define $W^3_- \lra Y^3 \times [0,1/2), \ldots, W^T_- \lra Y^T \times [0,1/2).$

We define $W^j_+ \lra Y^j \times (1/2 \times 1]$ in the same way to obtain
\begin{equation*}
	\sW = \{ W^1_- \lra Y^1 \times [0,1/2), W^1_+ \lra Y^1 \times (1/2, 1], \ldots, 
	W^T_- \lra Y^T \times [0,1/2), W^T_+ \lra Y^T \times (1/2, 1] \},
\end{equation*}
a self-dual mezzoperversity over $\sX.$

(In summary, we extend the metrics $g$ and $g'$ arbitrarily to an $\iie$ metric $G$ without changing them in collar neighborhoods of the boundary, and then we choose a Hodge mezzoperversity by extending the de Rham mezzoperversities trivially from $Y^i$ to $Y^i \times [0,1/2)$ on the left and from $Y^i$ to $Y^i \times (1/2, 1]$ on the right.)

Note the key point that since the strata induced by $Y^k \times [0,1/2)$ are disjoint from the strata induced by $Y^k \times (1/2,1],$ there is no compatibility required between the corresponding mezzoperversities.

Finally, define $R:\sX \lra M$ by $R(\zeta, t) = r(\zeta).$
The result is a cobordism
\begin{equation*}
	(\sX, G, \sW, R:\sX \lra M),
\end{equation*}
between $(\hat X,g,\cW, r:\hat X \lra M)$ and $(\hat X',g',\cW', r:\hat X \lra M).$
\end{proof}

This result suggests that we should define a seemingly coarser cobordism theory.
\begin{definition}
If $M$ is a topological space, we denote by $\Che{n}{M}$ the bordism group of 
pairs $(\hat X, r:\hat X \lra M)$ where $\hat X$ is an oriented Cheeger space and $r:\hat X \lra M$ is a continuous map smooth on $X.$

An admissible bordism between $(\hat X, r:\hat X\lra M)$ and $(\hat X', r':\hat X' \lra M)$ is a pair $(\sX, R:\sX \lra M)$ consisting of:\\
i) a smoothly stratified, oriented, compact pseudomanifold with boundary $\sX,$ whose boundary is $\hat X \sqcup \hat X',$
and whose strata near the boundary are collars of the strata of $\hat X$ or $\hat X',$\\
ii) a map $R: \sX \lra M$ that extends $r$ and $r'.$\\
We also require that $\sX$ is a `Cheeger space with boundary' in that it carries a self-dual mezzoperversity with 
a collared structure near its boundary.
\end{definition}

There is an obvious forgetful map $\cF:\Sig{n}{M} \lra \Che{n}{M}$ and the previous theorem shows that this map is an isomorphism.

\begin{corollary}\label{cor:LagIndep}
Every Cheeger space $\hat X$ has a well-defined analytic signature, $\sigma^{\an}(\hat X),$ equal to the index of the signature operator $(\eth_{\sign}, \cD_{\cW}(\eth_{\sign}))$ for any choice of self-dual mezzoperversity $\cW$ and adapted $\iie$ metric $g.$ The signature defines a homomorphism
\begin{equation*}
	\xymatrix @R=1pt
	{ \Che{n}{\pt} \ar[r]^-{\sigma^{\an}} & \bbZ \\
	[\hat X] \ar@{|->}[r] & \ind (\eth_{\sign}, \cD_{\cW}(\eth_{\sign})) }.
\end{equation*}

Moreover if $\Gamma$ is a finitely generated discrete group then for any smooth map $r:\hat X \lra B\Gamma$ there is a signature class in $\tK_0(C^*_r\Gamma;\bbQ)$ depending on no other choices. This signature class defines a group homomorphism
\begin{equation}\label{eq:ansymsign2}
	\xymatrix @R=1pt
	{ \Che{n}{B\Gamma} \ar[r]^-{\sigma^{\an}_{\Gamma}} & \tK_0(C^*_r\Gamma;\bbQ) \\
	[\hat X,r] \ar@{|->}[r] & \Ind(\eth_{\sign,\cW}^{\sG(r)}) }
\end{equation}
where $\cW$ is any self-dual mezzoperversity on $\hat X.$
\end{corollary}

We can combine this with Theorem \ref{thm:StratHomInvSign} to see that if $\hat X$ and $\hat M$ are Cheeger spaces and they are stratified homotopically equivalent (smoothly or continuously, by Theorem \ref{thm:SmoothApp}), then 
\begin{equation*}
	\sigma^{\an}(\hat X) = \sigma^{\an}(\hat M) \Min \bbZ.
\end{equation*}
Also, if $(\hat M, r:\hat M \lra B\Gamma) \in \Che{n}{B\Gamma}$ and $F \in \CI_{b,cod}(\hat X, \hat M)$ is a stratified homotopy equivalence
then 
\begin{equation}\label{eq:FullInvSignature}
	\sigma^{\an}_{\Gamma}(\hat M, r) = \sigma^{\an}_{\Gamma}(\hat X, F \circ r) \Min K_*(C^*_r\Gamma).
\end{equation}

\begin{remark}
It would be interesting to find a purely topological definition of $\sigma^{\an}_{\Gamma}$ as in the work of Banagl \cite{Banagl:msri} and Friedman-McClure \cite{Friedman-McClure:Sym} for Witt spaces; by Corollary \ref{cor:BordismMagic} such a signature would coincide with $\sigma^{\an}_{\Gamma}$ over $\bbZ[\tfrac12].$
\end{remark}

%%%%%%%%%%%%%%%%%%%%%%%%
\section{The higher analytic signatures of a Cheeger space} \label{sec:HigherSign}
%%%%%%%%%%%%%%%%%%%%%%%%

%%%%%%%%%%%%%%%%%%%%%%%%
\subsection{The analytic L-class of a Cheeger space}
%%%%%%%%%%%%%%%%%%%%%%%%

Following Thom \cite{Thom:Espaces, Thom:Classes}, Goresky-MacPherson \cite{GM1}, and Banagl \cite{Banagl:LClasses} we define 
the (analytic) L-class of a Cheeger space as an element in homology (see also the work of Cheeger \cite{Cheeger:Spec} on Witt spaces). 
The results of the previous section are used to show that this is independent of choices and intrinsic to the Cheeger space.

By means of \cite{Goresky:Thesis, Teufel} we may identify a smoothly stratified space $\hat X$ with a Whitney stratified subset of 
$\bbR^N$ for some $N\gg1.$  Fixing such an identification, and following \cite[\S5.3]{GM1}, we say that a continuous map 
$f:\hat X \lra \bbS^k$ 
% $\hat X$ to the $k$-dimensional sphere
% %
% \begin{equation*}
% \end{equation*}
% %
is transverse if:
\begin{enumerate}
\item[a)] $f$ is the restriction of a $\CI$ map $\wt f:\cU \lra \bbS^k$ for some neighborhood $\cU$ of $\hat X$ in $\bbR^N,$\\
\item[b)] $\wt f$ is transverse to the north pole $\cN \in \bbS^k,$\\
\item[c)] $\wt f^{-1}(N)$ is transverse to each stratum of $\hat X.$
\end{enumerate}
\begin{lemma}\label{lem:L} 
Let $\hat X$ be a Cheeger space and $k \in \bbN.$\\
i) If $f: \hat X \lra \bbS^k$ is a transverse map then $f^{-1}(\cN)$ is naturally a Cheeger space.\\
ii) There is a unique map
\begin{equation*}
	\theta:[\hat X,\bbS^k] \lra \bbZ,
\end{equation*}
where $ [\hat X, \bbS^{ 2 q}]$ denotes the set of homotopy classes of continuous maps,
which assigns the number $\sigma^{\an}(f^{-1}(\cN))$ to each transverse map $f:\hat X \lra \bbS^k.$
\end{lemma}

\begin{proof} $ $

i) Because $\wt f^{-1}(\cN)$ is a smooth submanifold of $\bbR^N$ that is transverse to all strata of $\hat X,$ we know from \cite[\S1.11]{GM:Morse} that
$f^{-1}(\cN) = \wt f^{-1}(\cN)\cap \hat X$ is naturally stratified with strata 
\begin{equation*}
	\{ f^{-1}(\cN)\cap Y : Y \text{ is a stratum of } \hat X \}
\end{equation*}
and that the inclusion $f^{-1}(\cN) \hookrightarrow \hat X$ is `normally non-singular'.
In particular, this means that $f^{-1}(\cN)$ has a tubular neighborhood in $\hat X$ that can be identified with a neighborhood of the zero section of the normal bundle of $\wt f^{-1}(\cN)$ in $\bbR^N,$ restricted to $\hat X.$
As this normal bundle is trivial (since it is induced from the normal bundle of $\cN \in \bbS^k$), $f^{-1}(\cN)$ has a neighborhood in $\hat X$ of the form $U \times f^{-1}(\cN)$ (cf. \cite[Proof of Theorem 5.1]{Moscovici-Wu}).

The projection map $U \times f^{-1}(\cN) \lra f^{-1}(\cN)$ is stratum preserving (where $U \times f^{-1}(\cN)$ is stratified by restricting the stratification of $\hat X$) by \cite[Theorem 1.11(4)]{GM:Morse} and so $f^{-1}(\cN)$ naturally inherits Thom-Mather data from $\hat X,$ and the link of a point in $f^{-1}(\cN)$ coincides with its link in $\hat X$ (see \cite[Proposition 5.2]{Goresky:Whitney} for a much more general case).
It also follows that any $\iie$ metric on $\hat X$ restricts to an $\iie$ metric on $f^{-1}(\cN)$ and both of these metrics induce the same metric on each link $Z^k,$ and that any mezzoperversity $\cW$ on $\hat X$ induces a mezzoperversity on $f^{-1}(\cN).$ 
Since self-duality of $\cW$ can be checked using the metric on each link, it follows that self-dual mezzoperversities induce self-dual mezzoperversities.
Thus $f^{-1}(\cN)$ is a Cheeger space.\\

ii) As in \cite[\S 5]{GM1}, standard techniques show that every continuous map $\hat X \lra \bbS^k$ may be approximated by a transverse map in the same homotopy class, and that between any two such maps there is a transverse homotopy $H:\hat X \times [0,1]\lra \bbS^k.$ Note that, by the same arguments as above, $H^{-1}(\cN)$ is a Cheeger space with boundary. This proves that the map $\theta$ can (only) be defined by choosing any transverse representative $f:\hat X \lra \bbS^k$ in a homotopy class $[\hat X, \bbS^k]$ and assigning to it  the number $\sigma^{\an}(f^{-1}(\cN)).$

\end{proof}

If $2 q$ is an even integer such that $4q > n+1$, then
\begin{equation*}
	\tH^{2 q}(\hat X; \bbQ) \cong [\hat X, \bbS^{ 2 q}] \otimes \bbQ
\end{equation*}
where $ [\hat X, \bbS^{ 2 q}]$ denotes the set of homotopy classes of continuous maps,
and so the map $\theta$ from the lemma induces a map:
\begin{equation*}
	H^{2 q}(\hat X; \bbQ)  \lra \bbQ
\end{equation*}
and so a class in $H_{2q}(\hat X; \bbQ).$
If $4 q \leq n+1,$ we pick an integer $\ell > n+1$ such that $4q + 4\ell > n + \ell +1,$ and then by K\"unneth we have $\tH^{2 q+2l }(\hat X \times \bbS^{2\ell} ; \bbQ) \simeq \tH^{2 q}(\hat X; \bbQ).$ Note that $\hat X \times \bbS^{2\ell}$ is a Cheeger space (with the product stratification, as is the product of $\hat X$ with any smooth manifold).

\begin{definition} \label{def:L}
The L-class  $\cL(\hat X)$ of a Cheeger space $\hat X$
is the rational homology class $\cL(\hat X) \in \tH_{\mathrm{even}}( \widehat{X} ; \bbQ)$ defined in the following way.
For any even integer $2 q$ such that $4 q > n+1$ and any smooth map $F: \widehat{X} \rightarrow  \bbS^{ 2 q}$ 
transverse to the North pole $N$, 
\begin{equation*}
	\langle \cL(\hat X) , [F] \rangle = \sigma^{\an}(F^{-1}(N))
\end{equation*}
where $[F]$ denotes the associated cohomology class in $\tH^{2q}(\widehat{X} ; \bbQ )$ and $\sigma^{\an}(F^{-1}(N))$ is the signature of the Cheeger space $F^{-1}(N).$
If $4 q \leq  n+1$ then one considers $\hat X \times \bbS^{2\ell}$  for $\ell \gg 1$ as explained above.
\end{definition}

%%%%%%%%%%%%%%%%%%%%%%%%%%%%
\subsection{The Chern character and the L-class}
%%%%%%%%%%%%%%%%%%%%%%%%%%%%

Let $(\hat X, g, \cW)$ be a Cheeger space with a self-dual mezzoperversity and an adapted $\iie$ metric.
We have seen that the associated signature operator defines a class in $\tK$-homology,
\begin{equation*}
	[\eth_{\sign,\cW}] \in \tK_*(\hat X)
\end{equation*}
and we have defined an $L$-class of $\hat X$ in the rational homology of $\hat X.$
In this section we extend a result of Cheeger \cite{Cheeger:Spec} and Moscovici-Wu \cite{Moscovici-Wu} relating these two objects, namely the Chern character
\begin{equation*}
	\Ch: K_*(\hat X) \otimes \bbQ \lra H_*(\hat X, \bbQ)
\end{equation*}
takes one to the other. We start by establishing a preliminary result on finite propagation speed that will be useful in the proof.

The next theorem states that the finite speed property continues to holds in our setting, despite the fact 
that the ideal boundary conditions are not preserved by the multiplication by Lipschitz functions. 
Note that the distance associated to an $\iie$ metric satisfies 
$$
\mathrm{dist}(p, p') = \sup | f(p)- f(p') |
$$ where $f$ runs over the subset of elements of  $\CI(\hat X)$ such that $|| g ( d f_{| X} ) ||_\infty \leq 1$.
 
\begin{theorem} \label{thm:fps} Let $dist$ denote the distance associated with the $\iie$ metric $g$ and consider  the de Rham operator with domain corresponding to a mezzoperversity $\cW,$ $(\eth_{\dR}, \cD_{\cW}(\eth_{\dR})).$ 
There exists $\delta_0 >0$ such that for any real $t \in [-\delta_0, \delta_0]$:
$$
{\rm Supp}\,  ( e^{ i t \,\eth_{\dR}})\, \subset \{ (p, p') \in \widehat{X}\times \widehat{X}\, |\, 
dist(p; p') \leq | t |\, \} \,.
$$
\end{theorem}
\begin{proof} 
Because $(\eth_{\dR}, \cD_{\cW}(\eth_{\dR}))$ is self-adjoint, the operator $e^{ i t \,\eth_{\dR}}$ defines for any real $t$ 
a bounded operator on  $H=L^2(X;\Lambda^* \Iie T^*X)$ which preserves $\cD_{\cW}(\eth_{\dR})$.
We follow Hilsum and his proof of Corollary 1.11 in \cite{Hilsum}. We can assume $t>0$ at the expense 
of working with  $-\eth_{\dR}$ instead of $\eth_{\dR}$.
Consider two Lipschitz functions $\phi, \psi $ on $\hat X$ with compact support 
such that $ || g( d \phi_{| X} )||_\infty \leq 1$, $ || g( d \psi_{| X} )||_\infty \leq 1$ and 
$$
dist ({\rm supp}\,(  \phi  ); {\rm supp}\,(  \psi  ) ) > t \, .
$$
We can assume that both $\phi$ and $\psi$ are compactly supported in a distinguished neighborhood $\cU_q\cong [0,\eps_0) \times \bbB^h \times Z_q$ of a point $q$ on a singular stratum. 
%$\phi$ (resp. $\psi$) has compact support included both in the domain of a 
%stratified chart (cf \eqref{eq:chart}) and in a $\eps_0-$neighborhood of a  stratum 
%(see just after \eqref{TubNhd}). 
It is clear that $\phi$ (resp. $\psi$) is a  uniform limit of a sequence $(\phi_k)$ (resp. $(\psi_k)$) 
of elements of $\CI(\hat X)$ satisfying the same properties and having support in a fixed compact subset of $\cU_q.$
%included both in the domain of a 
%stratified chart  and in a $\eps_0-$neighborhood of a  stratum.
Following Hilsum we consider the function $h_k$ defined on $\widehat{X}$ by:
$$
h_k(p) = {\rm inf} ( | t| , dist ( p , {\rm supp}\, \phi_k ) )- \psi_k (p) + {\rm inf}\, \psi_k ,\; {\rm if}\, p \in {\rm supp}\, \psi_k  
$$
$$ 
h_k(p) = {\rm inf} ( | t| , dist ( p , {\rm supp}\, \phi_k ) )- \phi_k (p) + {\rm inf}\, \phi_k ,\; {\rm if}\, p \in {\rm supp}\, \phi_k  
$$
$$
h_k(p)= {\rm inf} ( | t| , dist ( p , {\rm supp}\, \phi_k ) )\; {\rm otherwise}\,.
$$
\begin{lemma} There exists $\delta_0 \in (0, \eps_0]$ such that for each  $t \in (-\delta_0, \delta_0)$ and each  $k\in \N^*$, each stratum has a neighborhood in which 
$h_k$ does not depend on the link variable. 
\end{lemma}
\begin{proof} In a distinguished neighborhood of a point in a singular stratum we have coordinates $(x,y,z)$ such that the metric $g$ has the form:
$$
g= (d x)^2 + \phi^*g_Y(y) + x^2 h_Z(y)\,,
$$ where a choice of a connection is understood. 
Observe that this metric is bundlelike and that the fibers in $Z$ define a Riemannian foliation. Let 
$$ t\rightarrow (x,y,z ;\tau, \xi, \theta)(t)
$$ denote a geodesic in  the domain of this chart, where $t \in [0, a]$. 
If its tangent vector at $t=0$ is orthogonal to the fiber in $Z$ then its projection onto the basis 
will be a geodesic with the same length. One then gets easily the Lemma.
\end{proof}
Therefore, the multiplication by each $h_k$ preserves the domain $\cD_{\cW}$, and an immediate extension of Lemma 1.10 of \cite{Hilsum} 
shows that $\psi_k e^{ i t  \,\eth_{\sign,\cW}} \phi_k =0$ as an operator acting on $H$. Moreover, the two  sequences 
of multiplication operators $(\psi_k)_{k\in \N}$, $(\phi_k)_{k\in \N}$ converge in $B(H)$ respectively  
to the multiplication operators defined by $\psi$ and $\phi$. One 
 then gets immediately 
$\psi e^{ i t  \,\eth_{\sign,\cW}} \phi =0$ which proves the result. 
\end{proof}

\begin{theorem} \label{thm:eps}
{\item 1) } The operator  $\eth_{\dR}$ with domain $\cD_{\cW}(\eth_{\dR})$ admits an $\eps-$local  parametrix $G_1$
such that  $\eth_{\dR} G_1  - Id $ and $G_1 \eth_{\dR}   - Id $ are trace class.

{\item 2)} For any $h\in \cS( \R)$, the operator $h( \eth_{\dR})$ is trace class. 
\end{theorem}
\begin{proof} 
1) Note that \eqref{eq:trace} shows that  $G= \eth_{\dR}^{2n-1} (Id + \eth_{\dR}^{2n})^{-1}$ defines a parametrix 
for $\eth_{\dR}$ with compact remainder. Set for any real $s$,  $f(s)= \frac{s^{2n-1}}{1 + s^{2n}}$. 
We observe that $(i t)^{4n+1} \hat{f}(t)= \widehat{(\partial_s^{4n+1} f )}(t)$.
Therefore:
$$
G= \int_\R e^{ i t \eth_{\dR}} \hat{f}(t) \frac{d t}{  \sqrt{ 2 \pi}}\,. 
$$ Consider $\eps \in (0, \delta_0/2] $ and $\chi \in C^\infty( \R ; [0, 1])$ such that $\chi(t)=0$ for 
$t \in [-\eps , \eps]$ and $\chi(t)=1$ for 
$| t | \geq 2  \eps$.
Using integration by parts one checks easily that for any integer $k \geq 2$:
$$
\eth_{\dR}^k \int_\R e^{ i t \eth_{\dR}} (\chi(t) (i t) ^{-4n-1} ) \, ( i t)^{4n+1}  \hat{f}(t) \frac{d t}{  \sqrt{ 2 \pi}}
$$ is bounded on $L^2_{\iie}$. Now it follows from \cite[Theorem 4.5]{ALMP13.1} that
\begin{equation*}
	u, \eth_{\dR}^k u \in L^2(X; \Lambda^* \Iie T^*X) \implies u \in \rho^{1/2-\eps} H^k_{\iie}(X; \Lambda^* \Iie T^*X)
\end{equation*}
for any `total boundary defining function' $\rho$ and every $\eps>0,$ so one checks easily 
that $$
\int_\R e^{ i t \eth_{\dR}} \chi(t) \hat{f}(t) \frac{d t}{  \sqrt{ 2 \pi}}
$$ has range in $\cap_{\eps' >0} \,\rho^{1-\eps'} H^{2n}_{\iie}$. Therefore, using Theorem \ref{thm:fps} 
one sees that 
$$
G_1= \int_\R e^{ i t \eth_{\dR}} (1-\chi(t)) \hat{f}(t) \frac{d t}{  \sqrt{ 2 \pi}}
$$ does the job.

The proof of ($2$) is simpler than that of ($1$) and is left to the reader.
\end{proof} 
%We follow Hilsum and his proof of Corollary 1.11 in \cite{Hilsum}.
%Consider two Lipschitz functions $\phi, \psi $ on $\widehat{X}$ with compact support 
%such that $ || g( d \phi )||_\infty \leq 1$, $ || g( d \psi )||_\infty \leq 1$ and 
%$$
%d ({\rm support}\,(  \phi  ); {\rm support}\,(  \psi  ) ) > t \, .
%$$
%We can assume that  $\phi$ (resp. $\psi$) has compact support near a stratum. 
%It is clear that $\phi$ (resp. $\psi$) is a  uniform limit of a sequence $(\phi_k)$ (resp. $(\psi_k)$) 
%of elements of $\cA$ satisfying the same properties and having support in a fixed compact. 
%Following Hilsum we consider the function $h_k$ defined on $\widehat{X}$ by:
%$$
%h_k(p) = {\rm inf} ( | t| , d ( p , {\rm supp}\, \phi_k ) )- \psi_k (p) + {\rm inf}\, \psi_k ,\; {\rm if}\, p \in {\rm supp}\, \psi_k  
%$$
%$$ 
%h_k(p) = {\rm inf} ( | t| , d ( p , {\rm supp}\, \phi_k ) )- \phi_k (p) + {\rm inf}\, \phi_k ,\; {\rm if}\, p \in {\rm supp}\, \phi_k  
%$$
%$$
%h_k(p)= {\rm inf} ( | t| , d ( p , {\rm supp}\, \phi_k ) )\; {\rm otherwise}\,.
%$$

%Since each $h_k$ preserves the domain $\cD_{\cW}$, an immediate extension of Lemma 1.10 of \cite{Hilsum} 
%shows that $\psi_k e^{ i t  \,\eth_{\sign,\cW}} \phi_k =0$ as an operator acting on $H$. We then get immediately 
%$\psi e^{ i t  \,\eth_{\sign,\cW}} \phi =0$ which proves the result.
%\end{proof}

Theorems \ref{thm:fps} and \ref{thm:eps} allow to proceed as in Moscovici-Wu \cite{Moscovici-Wu} 
in order to prove the following
\begin{theorem}\label{thm:MoscWu}
The Chern character $\Ch: K_*(\hat X) \otimes \bbQ \lra H_*(\hat X, \bbQ)$ sends the
K-homology class of the signature operator $\eth_{\sign, \cW}$ to the analytic L-class in homology.
\end{theorem}
\begin{proof} We follow closely \cite{Moscovici-Wu}  and describe  only the steps that 
could exhibit a (modest) new difficulty.  
Let $\overline{u} \in C^\infty(\R)$ be an even function such that the function 
$\overline{v}(x)= 1 - x^2 \overline{u}(x)$ is Schwartz and both $\overline{u}$ and $\overline{v}$ have 
Fourier transforms supported in $(-1/4 , 1/4)$. There are smooth functions $u, v$ such that
$$
u(x^2)= \overline{u}(x),\; v(x^2)= \overline{v}(x)\,.
$$ Note that $v$ is Schwartz and so is the function
$$
w(x)= \frac{1-v(x)^2}{x} v(x)\,.
$$
For each real $t>0$, Theorems \ref{thm:fps} and \ref{thm:eps} allow to define  the following idempotent $$
P(t \eth_{\dR})= \begin{pmatrix}
	(v( t^2\eth_{\dR}^2))^2 \tau &  w( t^2\eth_{\dR}^2) \cdot t \eth_{\dR}  \tau  \\
	-v( t^2\eth_{\dR}^2) \cdot t \eth_{\dR}  \tau & (v( t^2\eth_{\dR}^2))^2 \tau
	\end{pmatrix} \,,
	$$ where $\tau$ denotes the grading associated with the Hodge star acting on differential forms 
	on $X$.
	Then we define an Alexander-Spanier cycle $\Lambda_* \{ t \eth_{\sign,\cW} \} $ 
	by setting for any real $t>0$:
	$$
	\Lambda_* \{ t \eth_{\sign,\cW} \} (f^0\otimes \ldots \otimes f^{2q} )=
	$$
	$$
	\frac{ (2 \pi i)^q}{ q! (2q+1) 2} {\rm Tr}\, \bigl( \sum_{\sigma \in \Sigma_{2q+1}} {\rm sign}\, (\sigma) \,
	P(t \eth_{\dR}) f^{\sigma(0)} \ldots P(t \eth_{\dR}) f^{\sigma(2q)} \, \bigr)\,;
	$$ where $f^0,\ldots, f^{2q} \in C(\widehat{X}),$ $q>0$ and $\Sigma_{2q+1}$ is the permutation 
	group of $\{0, \ldots , 2q\}$.
	
	In the case $q=0$ we set:
	$$
	\Lambda_* \{ t \eth_{\sign,\cW} \} (f^0) = {\rm Tr}\, \bigl(  P(t \eth_{\dR}) f^{0} - \begin{pmatrix} 
	\frac{1 -\tau}{2} & 0 \\
	0 &  \frac{1 -\tau}{2} \end{pmatrix} f^0\, \bigr)
	$$
	Thus $ \Lambda_* \{ t \eth_{\sign,\cW} \}$ defines  a skew-symmetric measure on $X^{2q+1}$ 
	supported in the $| t |-$neighborhood of the diagonal.
For any real $t>0$, $\Lambda_* \{ t \eth_{\sign,\cW} \} $  represents 
the homology class 	$ {\rm Ch} [\eth_{\sign,\cW}]$ ( see \cite[Section 4]{Moscovici-Wu}).

We now follow very closely the proof of Theorem 5.1 of \cite{Moscovici-Wu}. 
Consider an even integer $2 q$ such that $4 q > n +1$. Consider a map $F: \widehat{X} \rightarrow \bbS^{2q}$
transverse to the north pole $N$.
There exists (as above) an open neighborhood $U$ of $N$ in $\bbS^{2q}$ such that 
$F^{-1} (U) \simeq U \times F^{-1}(N) $ and $F$ is isomorphic to the projection to $U$ 
when restricted to $F^{-1} (U) $. 

Let $u \in H^{2q} (\bbS^{2q} ; \R)$ be such that $\langle u ; [ \bbS^{2q}] \rangle =1$.
We can represent $u$ by an Alexander-Spanier cocycle $\phi$ compactly supported inside $U$ (actually 
this is why we use Alexander-Spanier cohomology). Then $F^*(\phi)$, which represents the cohomology class 
$F^*(u)$,  is supported inside  $F^{-1} (U) $.
We choose an incomplete conic iterated metric on $ \bbS^{2q} \times F^{-1}(N) $ which coincides 
with the metric on $\widehat{X}$ in a neighborhood of the (compact) support of $F^*(\phi)$. 
According to Lemma \ref{lem:L},  the stratified space $F^{-1}(N) $ (and thus also $ \bbS^{2q} \times F^{-1}(N) $)  admits 
a self-dual mezzoperversity $\wt\cW$ which is induced by the one of $\widehat{X}$.
Therefore we have a signature operator $\wt \eth_{\sign,\wt\cW}$ on $ \bbS^{2q} \times F^{-1}(N) $. 
An immediate extension of Lemma 4.2 of  \cite{Moscovici-Wu} shows, thanks to the finite propagation
speed property,  that for $t>0$ small enough:
$$
\langle \Lambda_* \{ t \eth_{\sign,\cW} \} ; F^* (u) \rangle = \langle \Lambda_* \{ t \wt \eth_{\sign,\wt\cW} \}; \pi^*(\phi)\rangle \,, 
$$ where $\pi: \bbS^{2q} \times F^{-1}(N) \rightarrow  \bbS^{2q}$ denotes the projection.

Now, we use the proof of Lemma 4.3 of \cite{Moscovici-Wu} (multiplicative formula).
We get:
$$
\langle \Lambda_* \{ t \wt \eth_{\sign,\wt \cW} \}; \pi^*(\phi)\rangle = 
 (\, {\rm ind}\, \eth^{F^{-1}(N)}_{\sign,\wt \cW} \,)  \,2^q \langle [\bbS^{2q}] , [u] \rangle\,,
$$ which proves the result.
\end{proof}

We point out that, in particular, this shows that the Chern character of the rational K-homology class $[\eth_{\sign,\cW}]$ is independent
of the metric and of the choice of $\cW$.

%\begin{corollary}
%The rational K-homology class $[\eth_{\sign,\cW}] \in \tK_*(\cC(\hat X)) \otimes \bbQ$ is independent
%of the metric and of the choice of $\cW$.
%\end{corollary}
%
%\begin{proof}
%The homology Chern character of Moscovici-Wu is a rational isomorphism (?) and the $L$-class is by construction (and Corollary \ref{cor:LagIndep}) independent of these choices.
%\end{proof}

%%%%%%%%%%%%%%%%%%%%%%%%
\subsection{Higher signatures of a Cheeger space and the Novikov conjecture}
\label{HigherSign} $ $\\
%%%%%%%%%%%%%%%%%%%%%%%%

\begin{definition}
Let $\hat X$ be a Cheeger space.
Let $\Gamma = \pi_1\hat X$ and $r:\hat X \lra B\Gamma$ a classifying map for the universal cover. The higher signatures of $\hat X$ are the collection of rational numbers
\begin{equation*}
	\lrbrac{ \ang{ \alpha, r_*\cL(\hat X) }: \alpha \in \tH^*(B\Gamma, \bbQ) }
\end{equation*}
\end{definition}

From Corollary \ref{cor:assembly}, the reduction of the Novikov conjecture for the higher signatures of $\hat X$ is 
the strong Novikov conjecture for $\pi_1 X$ exactly as in the classical case (see \cite[Theorem 11.1]{ALMP11}). 
Indeed, recall that the strong Novikov conjecture states that the assembly map 
\begin{equation*}
	\beta:\tK_*(B\pi_1 \hat X) \lra \tK_*(C^*_r\pi_1 \hat X)
\end{equation*}
is rationally injective.

\begin{theorem}
%Let $\hat X$ be a Cheeger space.
%If the strong Novikov conjecture is true for $\pi_1\hat X,$ then the higher signatures of $\hat X$ are stratified homotopy invariants.
%
Let $\hat X$ be a Cheeger space whose fundamental group $\pi_1\hat X$ satisfies the strong Novikov conjecture.
If $\hat M$ is a Cheeger space stratified homotopy equivalent to $\hat X$ (smoothly or continuously), then their higher signatures coincide. 
\end{theorem}

\begin{proof}
The claim is that if $\hat X$ and $\hat M$ are stratified homotopy equivalent Cheeger spaces with fundamental group $\Gamma,$ 
then 
\begin{equation*}
	\ang{ \alpha, r^X_*\cL(\hat X) }
	= \ang{ \alpha, r^M_*\cL(\hat M) }
	\Mforall \alpha \in \tH^*(B\Gamma, \bbQ)
\end{equation*}
where $r^M:\hat M \lra B\Gamma$ and $r^X:\hat X \lra B\Gamma$ are classifying maps for the universal covers. Equivalently, that
\begin{equation}\label{eq:NovikovEquiv}
	r^X_*\cL(\hat X)  = r^M_*\cL(\hat M) \Min \tH_*(B\Gamma; \bbQ).
\end{equation}

By Theorem \ref{thm:SmoothApp} there is a smooth stratified homotopy equivalence $F \in \CI_{b,cod}(\hat X, \hat M).$ 
We know from \eqref{eq:FullInvSignature} that 
\begin{equation*}
	\sigma^{\an}_{\Gamma}(\hat M, r) = \sigma^{\an}_{\Gamma}(\hat X, F \circ r),
\end{equation*}
and hence from Corollary \ref{cor:assembly} that for any self-dual mezzoperversity on $\hat M,$ $\cW_M,$ 
\begin{equation*}
	\beta(r^M_*[\eth_{\sign, \cW_M}]) = \beta((F \circ r^M)_*[\eth_{\sign,F^{\sharp}\cW_M}]).
\end{equation*}
Since we are assuming the strong Novikov conjecture holds for $\Gamma,$ this implies
\begin{equation*}
	r^M_*[\eth_{\sign, \cW_M}] = (F \circ r^M)_*[\eth_{\sign,F^{\sharp}\cW_M}].
\end{equation*}
Now taking Chern characters and using Theorem \ref{thm:MoscWu}, we see that
\begin{equation*}
	r^M_*\cL(\hat M) = (F \circ r^M)_*\cL(\hat X)   \Min \tH_*(B\Gamma; \bbQ).
\end{equation*}
But $F \circ r^M$ is a classifying map for the universal cover of $\hat X$ so this proves \eqref{eq:NovikovEquiv}.

\end{proof}

%%%%%%%%%%%%%%%%%%%%%%%%
\section{Relation with the topological signature} \label{sec:RefTop} $ $\\
%%%%%%%%%%%%%%%%%%%%%%%%
The analytic treatment of the {\em signature} operator developed in \cite{ALMP13.1} has a topological analogue developed earlier by Markus Banagl \cite{BanaglShort}, which has served as inspiration for the analytic development.
On an arbitrary {\em topologically} stratified non-Witt space (see \cite[Definition 4.1.1]{BanaglLong}), $\hat X,$ Banagl  defines a category $SD(\hat X)$ of `self-dual sheaves' as a way of extending Poincar\'e Duality.
In \cite{ABLMP} the authors, together with Banagl, develop a topological analogue of the analytic treatment of the $L^2$-cohomology from \cite{ALMP13.1}.

A topologically stratified space $\hat X$ is called an {\bf $L$-space} if $SD(\hat X) \neq \emptyset.$ In \cite[Proposition 4.3]{ABLMP} we show that a smoothly stratified space is a Cheeger space if and only if it is an $L$-space. Moreover in that case a self-dual mezzoperversity corresponds to Banagl's `Lagrangian structures.'
Every $L$-space $\hat X$ has a signature which we will denote
\begin{equation*}
	\sigma^{\mathrm{top}}(\hat X).
\end{equation*}
It is defined using a self-dual sheaf complex, but Banagl showed that it depends only on $\hat X.$ If $\hat X$ admits a smooth stratification then it also has the analytic signature defined above, 
\begin{equation*}
	\sigma^{\mathrm{an}}(\hat X).
\end{equation*}
We will show that these signatures coincide.\\

As mentioned above, our cobordism groups $\Sig{n}{\pt}$ are smooth analogues of topological groups first defined in \cite{BanaglShort}, which we will denote by $\Sigt{n}{\pt}.$ A class in this group is represented by a pair $(\hat X, \sc S)$ where $\hat X$ is an $n$-dimensional topologically stratified pseudomanifold and $\sc S$ is a self-dual sheaf complex. An admissible bordism is a compact oriented topologically stratified pseudomanifold with boundary $\sX$ together with a sheaf complex in $SD(\sX)$ which pushes to the given sheaf complexes over the boundary. Minatta has computed these groups \cite[Proposition 3.8]{Minatta} (see also \cite[\S 4]{Banagl:Computing})
by showing that the signature is a complete invariant,
\begin{equation}\label{eq:BanaglComp}
	\Sigt{n}{\pt}  \cong
	\begin{cases}
	\bbZ & \Mif n \text{ is a multiple of four}\\
	0 & \text{ otherwise }
	\end{cases}
\end{equation}

The proof of \eqref{eq:BanaglComp} is to note that if $\hat X$ is an $L$-space with vanishing signature, then the cone over $\hat X$ is  a compact oriented topologically stratified pseudomanifold with boundary and the self-dual sheaf over $\hat X,$ pulled-back to the interior of the cone, extends to a self-dual sheaf over the entire cone.
Thus the topological signature, which is obviously surjective since smooth manifolds are $L$-spaces, is also injective. In particular, every $L$-space with a non-zero signature is topologically cobordant to a smooth manifold with that signature.

Notice that we can carry out the same proof to compute the group $\Sig{n}{\pt},$ using either the analytic or topological signatures. Indeed, either of these signatures is surjective since smooth manifolds are Cheeger spaces.
The proof of injectivity extends to the smooth context because the cone over a smoothly stratified space admits a smooth stratification. So a Cheeger space with vanishing topological or analytic signature is smoothly nullcobordant. If either the topological or analytic signature is non-vanishing then the Cheeger space is smoothly cobordant to a smooth manifold with that signature. But since the smooth and topological signatures coincide on smooth manifolds, they must coincide on all Cheeger spaces. Therefore, one gets:

\begin{theorem}
The analytically defined signature of a Cheeger spaces coincides with its topological signature
\begin{equation*}
	\sigma^{\an}(\hat X)=
	\sigma^{\tp}(\hat X)
\end{equation*}
 and induces an isomorphism $\Sig{n}{\pt} \cong \Sigt{n}{\pt}.$
\end{theorem}

We defined an analytic $L$-class above using the analytic signature of submanifolds transverse to the stratification.
There is also a topological $L$-class. Indeed, it follows from \cite{CSW} that every self-dual complex of sheaves has an $L$-class in homology which is a topological invariant.
Banagl showed that on an $L$-space, $\hat X,$ the $L$-class is independent of the choice of self-dual sheaf complex from $SD(\hat X)$ and is defined in the same way as we defined the analytic $L$-class  (Def. \ref{def:L}) but using the topological signature. Hence we have:

\begin{corollary}
The analytic $L$-class of a Cheeger spaces coincides with its topological $L$-class.
\end{corollary}

Minatta \cite[Theorem 3.4]{Minatta} has shown that $M \mapsto \Sigt{*}{M}$ is a multiplicative generalized homology theory (with the product induced by Cartesian product) and by adding a formal variable $t,$ the signature becomes an isomorphism of graded rings
\begin{equation*}
	\xymatrix @R=1pt{ 
	\Sigt{*}{\pt} \ar[r]^{\cong} & \bbZ[t] \\
	[\hat X] \ar@{|->}[r] & \sigma^{\tp}(\hat X) t^{\dim X/4} }
\end{equation*}
It is easy to see that $M \mapsto \Sig{*}{M}$ is also a multiplicative generalized homology theory. (Indeed, Minatta's proof holds in the smooth setting and even simplifies, as his main tool is a transversality theorem which is standard for smooth spaces but required a `very sophisticated argument' to establish in the topological category.)
Moreover, by Theorem 5.7 in \cite{ABLMP}, there is a  natural map $\Sig{*}{M} \lra \Sigt{*}{M}$ 
which induces  a natural transformation of multiplicative homology theories. Since we have seen that this map is an isomorphism when $M = \pt,$ we conclude (see e.g., \cite[Theorem 2.9]{Kono-Tamaki}):

\begin{corollary}
There is an isomorphism of multiplicative generalized homology theories (over $CW$-complexes)
\begin{equation*}
	\Sig{*}{-} \lra \Sigt{*}{-}.
\end{equation*}
\end{corollary}

Let us recall  Minatta's  computation of  this homology theory with  coefficients in the following three cases:

a] For rational coefficients \cite[Proposition 4.2]{Minatta},
\begin{equation*}
	\Sigt{*}{M} \otimes_\bbZ \bbQ \cong \tH_*(M; \bbQ[t]) \,.
\end{equation*}

b] At odd primes \cite[Proposition 4.6]{Minatta}, \cite[Theorem 4.1]{Banagl:Computing},
\begin{equation*}
	\Sigt{*}{M} \otimes_{\bbZ} \bbZ[\tfrac12] \cong \Omega^{\mathrm{SO}}_*(M) \otimes_{\Omega^{\mathrm{SO}}_*(\pt)} \bbZ[\tfrac12][t],
\end{equation*}
where $\Omega^{\mathrm{SO}}_*$ denotes the smooth oriented bordism.

c]  At two \cite[Corollary 4.8]{Minatta},
\begin{equation*}
	\Sigt{*}{M} \otimes \bbZ_{(2)} \cong \tH_*(M;\bbZ_{(2)}[t]).
\end{equation*}

In particular, for any topological space $M$ there is a surjection \cite[Corollary 4.1]{Banagl:Computing}
\begin{equation*}
	\Omega^{\mathrm{SO}}_*(M) \otimes_{\bbZ} \bbZ[\tfrac12][t] \lra \Sigt{*}{M} \otimes_{\bbZ} \bbZ[\tfrac12].
\end{equation*}
This allows us to state a uniqueness result for the `analytic symmetric signature' of \eqref{eq:ansymsign}, \eqref{eq:ansymsign2}, $\sigma^{\an}_{\Gamma}(\cdot)$.
Note that, by the work of Kasparov and Mishchenko, for a smooth closed manifold $\Ind(\eth_{\sign}^{\sG(r)})$ coincides with the symmetric signature of Mishchenko.

\begin{corollary}\label{cor:BordismMagic}
Any homomorphism 
\begin{equation*}
	\Sig{*}{B\Gamma} \otimes_{\bbZ} \bbZ[\tfrac12] \lra K_*(C^*_r\Gamma; \bbQ)
\end{equation*}
that coincides with the symmetric signature of Mishchenko, tensored with $\bbZ[\tfrac12],$ for smooth manifolds
is equal to $\sigma^{\an}_{\Gamma}(\cdot) \otimes \bbZ[\tfrac12].$
\end{corollary}

The proof is analogous to the corresponding statement for Witt spaces, \cite[Proposition 11.1]{ALMP11}.

%\bibliography{Wittless}
%\bibliographystyle{amsalpha}

\end{document}